\documentclass[12pt]{amsart}
\usepackage{amsfonts}
\usepackage{amsrefs}
\usepackage{amssymb}
\usepackage{amsmath}
\usepackage{amsthm}
\usepackage{verbatim}
\usepackage{esint} 
\usepackage{booktabs}
\usepackage{color}

\newcommand{\delo}{\partial \Omega}
\renewcommand{\div}{\mathrm{div}}

\newcommand{\dom}{\partial\Omega}

\newcommand{\R}{{\mathbb R}}

\newcommand{\diff}{\ensuremath{\:\text{d}}}

\def\vint{\mathop{\mathchoice%
          {\setbox0\hbox{$\displaystyle\intop$}\kern 0.22\wd0%
           \vcenter{\hrule width 0.6\wd0}\kern -0.82\wd0}%
          {\setbox0\hbox{$\textstyle\intop$}\kern 0.2\wd0%
           \vcenter{\hrule width 0.6\wd0}\kern -0.8\wd0}%
          {\setbox0\hbox{$\scriptstyle\intop$}\kern 0.2\wd0%
           \vcenter{\hrule width 0.6\wd0}\kern -0.8\wd0}%
          {\setbox0\hbox{$\scriptscriptstyle\intop$}\kern 0.2\wd0%
           \vcenter{\hrule width 0.6\wd0}\kern -0.8\wd0}}%
          \mathopen{}\int}

\normalbaselines 

\newtheorem{theorem}{Theorem}[section]
\newtheorem{lemma}[theorem]{Lemma}

\newtheorem{proposition}[theorem]{Proposition}
\newtheorem{definition}{Definition}[section]

\begin{document}
\title[Operators satisfying Carleson condition]{Regularity and Neumann problems for operators with real coefficients satisfying Carleson conditions}

\author{Martin Dindo\v{s}}
\address{School of Mathematics, \\
The University of Edinburgh and Maxwell Institute of Mathematical Sciences, Edinburgh UK}
\email{M.Dindos@ed.ac.uk}

\author{Steve Hofmann}
\address{Dept. of Mathematics,\\
University of Missouri, US}
\email{hofmanns@missouri.edu}
\thanks{The second author was supported by NSF grant number DMS-2000048.}

\author{Jill Pipher}
\address{Dept. of Mathematics, \\ 
 Brown University, US}
\email{jill\_pipher@brown.edu}

\begin{abstract}
In this paper, we continue the study of a class of second order elliptic operators of the form $\mathcal L=\div(A\nabla\cdot)$ in a domain above a Lipschitz graph in 
$\mathbb R^n,$
where the coefficients of the matrix $A$ satisfy a Carleson measure condition, expressed as a condition on the oscillation on Whitney balls. For this class of operators, it 
is known (since 2001) that the $L^q$ Dirichlet problem is solvable for some $1 < q < \infty$. Moreover, further studies completely
resolved the range of $L^q$ solvability of 
the Dirichlet, Regularity, Neumann problems in Lipschitz domains, when the Carleson measure norm of the oscillation is sufficiently small. 

We show that there exists $p_{reg}>1$ such that for all $1<p<p_{reg}$ the $L^p$ Regularity problem for the operator $\mathcal L=\div(A\nabla\cdot)$ is solvable. Furthermore $\frac1{p_{reg}}+\frac1{q_*}=1$ where $q_*>1$ is the number such that the $L^q$ Dirichlet problem for the adjoint operator $\mathcal L^*$ is solvable for all $q>q_*$. 

Additionally when $n=2$, there exists $p_{neum}>1$ such that for all $1<p<p_{neum}$ the $L^p$ Neumann problem for the operator $\mathcal L=\div(A\nabla\cdot)$ is solvable. Furthermore $\frac1{p_{reg}}+\frac1{q^*}=1$ where $q^*>1$ is the number such that the $L^q$ Dirichlet problem for the operator $\mathcal L_1=\div(A_1\nabla\cdot)$ with matrix $A_1=A/\det{A}$ is solvable for all $q>q^*$. 
\end{abstract}

\maketitle

\section{Introduction}
In this paper, we consider the solvability of certain boundary value problems (Regularity, Neumann) for a class of 
elliptic second order divergence form equations with real coefficients satisfying a natural and well-studied Carleson measure condition. Some of
the extensive literature in this subject includes  \cite{ Dsystems, DHM, DLP, DPP, DR, DPR, HMMTZ, HLM, HMM, KP}.

The operators we study have the form  $L= \mbox{div}(A
\nabla)$ where the matrix $A(X)=(a_{ij}(X))$ is uniformly elliptic in the
sense that there exists a positive constant $\Lambda$ such that
$$
\Lambda|\xi|^2 \le \sum_{i,j} a_{ij}(X)\xi_i\xi_j \le
\Lambda^{-1}|\xi|^2,
$$
for all $X$ and all $\vec{\xi}\in \R^n$. 
The matrix $A$ is not assumed to be symmetric, and the non-symmetry imposes constraints on both the methodology and on the range of the 
sharp results.
The class of operators we study have coefficients that are assumed to satisfy the following condition on gradient of $A$, or a related weaker oscillation condition,
 namely that
\begin{equation}\label{CarlX}
d\mu(X)=|\nabla A(X)|^2\delta(X)\qquad\mbox{is a Carleson measure, and} \qquad \delta(X) |\nabla A(X)| \le C,
\end{equation}
where $\delta(X)$ is the distance of a point $X\in\Omega$ to the boundary. 
\smallskip

This condition first appeared  in a conjecture of Dahlberg, where it arose via a specific change of variables \eqref{CHV} introduced by Dahlberg, Kenig-Stein, and by Ne\v{c}as
(\cite{Dahl}, \cite{N}). This change of variable is relevant to the study of smooth equations in Lipschitz domains. In particular,  an example of an operator satisfying the above Carleson condition is the divergence
form operator generated from the pull-back of the Laplacian under the change of variable in \eqref{CHV}.

There are three well-studied boundary value problems that have natural physical interpretations for these elliptic equations: Dirichlet, Regularity, and Neumann. 
The Dirichlet problem prescribes boundary data of the solution in a domain, Regularity prescribes the tangential derivative of the boundary data as well, and Neumann
prescribes the normal (or co-normal) derivative of the solution. The precise definition of these problems, for data in Lebesgue spaces, will be given in section 2.

In
\cite{KP}, the methods of \cite{KKPT} were applied to study the Dirichlet problem for a class of divergence form operators satisfying \eqref{CarlX}.
There is a constant associated with the 
``size" of this Carleson measure, referred to as the Carleson norm. See section 2 for the definition of Carleson measure, and its norm.  
This Carleson condition on the coefficients of the matrix implies that $|\nabla a_{i,j}(X)|$ is bounded {\it away} from the boundary, but could blow up near the boundary.

The main result of \cite{KP} is that, for an operator in this class, defined in a Lipschitz domain, the elliptic measure and surface Lebesgue measure on such domains are mutually absolutely continuous. 
Moreover, there exists a $p < \infty$ such that the Dirichlet problem with boundary data in the Lebesgue $L^p$ space with respect to 
surface measure is solvable. 

In \cite{DPP}, it was shown that the range of solvability in $L^p$ of the Dirichlet problem for these operators could be extended to 
the full range of $1<p<\infty$ if the Carleson norm of the matrix coefficients was sufficiently small. 
Operators whose coefficients satisfy this {\it small Carleson norm} condition also arise naturally. For example, consider the Laplacian in the region above a graph $t = \varphi(x)$.
If the function $\varphi$ has the property that  $\nabla \varphi \in L^{\infty} \bigcap$ {\it
VMO}, a weaker property than $\varphi \in C^1$, the solvability of the Dirichlet problem in $L^p$ for $1 < p < \infty$ is a corollary of the main theorem of \cite{DPP}
via a change of variable. The function space $VMO$, introduced by Sarason (\cite{Sa}) consists of those $BMO$ (bounded mean oscillation) functions that can be approximated in $BMO$ norm by $C^\infty$ functions.

The Regularity and Neumann problems have also been studied for the class of elliptic operators whose coefficients satisfy \eqref{CarlX}, but only with additional assumptions.
If the Carleson norm of the expression in \eqref{CarlX}  is sufficiently small, both of these boundary value problems 
were shown to be solvable in the full range of $1<p<\infty$, first in two dimensions in 
\cite{DR}, and later in all dimensions in \cite{DPR}.

A central problem to complete the study of these
operators in Lipschitz domains (or even smooth domains) has been open since the 2001 results of \cite{KP}: the solvability of the Regularity and Neumann problems when the Carleson measure norm is merely finite, as opposed to small.

In this paper, we fully resolve solvability of Regularity problems in all dimensions on Lipschitz graph domains, as well as the Neumann problems in two dimensions. In fact, we give two different proofs of solvability of the two-dimensional Regularity problem.

Solvability in two dimensions, and subsequent passage from two dimensions to all dimensions, has been quite typical for progress in this subject.
One example, from the two-dimensional work of \cite{DR} to the higher dimensional generalizations \cite{DPR} was mentioned earlier. In another example, the passage to higher dimensions 
required the development
of major new ideas: the solvabilty
of the Dirichlet and Regularity problems in \cite{HKMP1} and \cite{HKMP2} came fifteen years after the two-dimensional results of \cite{KKPT2}, and used
the toolbox developed in connection with the solution to Kato's conjecture. The innovation in two dimensions that is applicable to the problem we solve
in this paper is a change of variable
observed by Feneuil in \cite{F} in another context, and is described in Section 4.
In Section 5, we present the new idea that leads to solvability of the Regularity problem in all dimensions $n \geq 2$.

Importantly, condition \eqref{CarlX} can be replaced by a weaker \lq\lq oscillation" condition of coefficients on Whitney balls; we will state our main result in this form.

\begin{theorem}\label{RNduality}Let $\Omega$ be 
an unbounded Lipschitz domain in $\mathbb R^n$, $n\ge 2$ of the form $\{(x,t):\, t>\phi(x)\}$ for some Lipschitz function $\phi:\mathbb R\to\mathbb R$.
Let $A:\Omega\to M_{n\times n}(\mathbb R)$ be a real matrix valued functions such that for some $\lambda,\Lambda>0$ we have
\begin{equation}\label{Ellip}
\langle A(X)\xi,\xi\rangle\ge\lambda|\xi|^2,\qquad \left|\langle A(X)\xi,\eta\rangle\right|\le\Lambda|\xi||\eta|,\quad\mbox{for all }\xi,\eta\in\mathbb R^n
\end{equation}
and a.e. $X\in\Omega$.

Suppose that $A$ satisfies a Carleson condition on oscillation on coefficients in Whitney balls in $\Omega$, that is
\begin{equation}\label{CC}
\delta(X)^{-1}\left[\sup_{Y,Z\in B(X,\delta(X)/2)}|A(Y)-A(Z)|\right]^2\mbox{ is a Carleson measure.}
\end{equation}

Then there exists $p_{reg}>1$ such that for all $1<p<p_{reg}$ the $L^p$ Regularity problem for the operator $\mathcal L=\div(A\nabla\cdot)$ is solvable. Furthermore $\frac1{p_{reg}}+\frac1{q_*}=1$ where $q_*>1$ is the number such that the $L^q$ Dirichlet problem for the adjoint operator $\mathcal L^*$ is solvable for all $q>q_*$. 

Additionally when $n=2$, there exists $p_{neum}>1$ such that for all $1<p<p_{neum}$ the $L^p$ Neumann problem for the operator $\mathcal L=\div(A\nabla\cdot)$ is solvable. Furthermore $\frac1{p_{reg}}+\frac1{q^*}=1$ where $q^*>1$ is the number such that the $L^q$ Dirichlet problem for the operator $\mathcal L_1=\div(A_1\nabla\cdot)$ with matrix $A_1=A/\det{A}$ is solvable for all $q>q^*$. 
\end{theorem}

We remark that, for simplicity, we have stated the result for domains of the form $\{(x,t):\, t>\phi(x)\}$, but the arguments should carry over to the setting of bounded Lipschitz domains. This will require some further localization results which are
not addressed in this paper.

One of the main tools we use to prove Theorem \ref{RNduality} is Theorem \ref{t1} below, which reduces the solvability of the Regularity problem for matrices whose coefficients satisfy condition (\ref{CC})
to solvability of the Regularity problem for an operator defined by a block-form matrix. This
theorem holds in all dimensions. Its proof, in Section 3, can be read independently of its application to the solution of Regularity and Neumann problems which are located in Sections 4 and 5.

\begin{theorem}\label{t1} Let $\mathcal L=\div(A\nabla\cdot)$ be an operator in ${\mathbb R}^n_+$ where matrix $A$ is uniformly elliptic, with bounded real coefficients such that there exists a constant $C$
\begin{equation}\label{Carl2}
|\nabla A|^2t\,dt\,dx\qquad\mbox{is a Carleson measure, and} \qquad t |\nabla A| \le C.
\end{equation}
Suppose that for some $p>1$ the $L^p$ Regularity problem for the block form operator
\begin{equation}\label{e0a}
{\mathcal L}_0 u= \mbox{\rm div}_\parallel(A_\parallel \nabla_\parallel u)+u_{tt},
\end{equation}
(where $A_{\parallel}$ is the matrix $(a_{ij})_{1\le i,j\le n-1}$) is solvable in ${\mathbb R}^n_+$. \vglue3mm

Then we have the following: For any $1<q<\infty$ the $L^q$ Regularity problem for the operator $\mathcal L$ is solvable in ${\mathbb R}^n_+$ if and only if the $L^{q'}$ Dirichlet problem for the adjoint operator $\mathcal L^*$ is solvable in ${\mathbb R}^n_+$.
\end{theorem}

One direction of the equivalence in this theorem has been proven without any of the stated assumptions \eqref{Carl2} or \eqref{e0a}. Namely, the solvability
of the $L^q$ Regularity problem for the operator $\mathcal L$ implies the solvability of the $L^{q'}$ Dirichlet problem for the adjoint operator $\mathcal L^*$ for
$q' = q/(q-1).$ (c.f. \cite{DK, KP1}). Thus the novelty in Theorem \ref{t1} is the converse direction.

Finally, in Section 5 we prove that the Regularity problem for operators \eqref{e0a} is solvable under the following assumptions:

\begin{theorem}\label{tblock}
Let ${\mathcal L}_0 u= \mbox{\rm div}_\parallel(A_\parallel \nabla_\parallel u)+u_{tt}$ be an operator in ${\mathbb R}^n_+$ where matrix $A_\parallel$ is uniformly elliptic $(n-1)\times(n-1)$ matrix, with bounded real coefficients such that \begin{equation}\label{Carl2m}
d\mu(X)=\delta(X)\left[\sup_{B(X,\delta(X)/2)}|\nabla A(X)|\right]^2\,dX\qquad\mbox{is a Carleson measure.}
\end{equation}
Then we have the following: For any $1<q<\infty$ the $L^q$ Regularity problem for the operator $\mathcal L_0$ is solvable in ${\mathbb R}^n_+$.
\end{theorem} 

\medskip

The reduction to block form is a special feature of working in the Lipschitz graph domain, permitting us to exploit a ``preferred direction" and introduce Riesz transform type operators via an integration in that preferred direction. This was also a feature of the main result of \cite{HKMP2}, where solvability of the Regularity problem for so-called
``$t$-independent" elliptic divergence form operators with data in $L^p$ (for some possibly large value of $p$) was established. The methods of this paper are similarly reliant on working in domains that are locally a graph, and making use of that graph direction.
\smallskip

As we were completing this manuscript, we learned that M. Mourgoglou, B. Poggi, and X. Tolsa (\cite{MPT}) were simultaneously completing a manuscript that
showed solvability of the Regularity problem for this same class of elliptic operators, but on uniformly rectifiable domains. Their methods, necessitated 
by the weaker geometric assumptions on the domain, are very different from those in this paper. In particular, they use the results of \cite{DPR} together with a new corona decomposition introduced in \cite{MT}. While their end result is more general - due to the weaker assumptions on the domain - the ideas and techniques we present here are novel in the context of Lipschitz domain theory, give a significantly shorter and self-contained argument, and serve to illuminate the specific additional technicalities that involve passage from local Lipschitz graphs to more general domains.

\section{Background and definitions}

In this section, we will state the relevant definitions and background for domains $\Omega = \{(x,t): t > \varphi(x)\}$ where $ \varphi(x):\mathbb R^n \rightarrow \mathbb R$ is a Lipschitz function.

\begin{definition} Let $\Omega$ be as above.
For $Q\in\partial\Omega$, $X\in \Omega$ and $r>0$ we write:
\begin{align*}
\Delta_r(Q) &= \partial \Omega\cap B_r(Q),\,\,\,\qquad T(\Delta_r) = \Omega\cap B_r(Q),\\
\delta(X) &=\text{\rm dist}(X,\partial \Omega).
\end{align*}
\end{definition}

\begin{definition}\label{cmeasure}
Let $T(\Delta_r)$ be the Carleson region associated to a surface
ball $\Delta_r$ in $\delo$, as defined above. A measure $\mu$ in $\Omega$ is
Carleson if there exists a constant $C$ such that 
\[\mu(T(\Delta_r))\le C \sigma (\Delta_r).\]
The best possible $C$ is the Carleson norm and will denoted by $\|\mu\|_{Carl}$. The notation $\mu \in \mathcal C$
means that the measure  $\mu$ is Carleson. \vglue1mm 
\end{definition}

\begin{definition}\label{dgamma}
A cone of aperture $a > 0$ is a non-tangential approach region for $Q
\in \partial \Omega$ of the form
\[\Gamma_{a}(Q)=\{X\in \Omega: |X-Q|\le (1+a)\;\;{\rm dist}(X,\partial \Omega)\}.\]
\end{definition}

For ease of notation, and when there is no need for the specificity, we shall omit the dependence on the aperture of the cones in the definitions of the square function and nontangential maximal 
functions below.

\begin{definition}
The square function of a function $u$ defined on $\Omega$, relative to the family of cones $\{\Gamma(Q)\}_{Q \in \partial\Omega}$, is
\[S(u)(Q) =\left( \int_{\Gamma(Q)} |\nabla u (X)|^2
\delta(X)^{2-n} dX \right)^{1/2}\] at each $Q \in \partial
\Omega$. The
non-tangential maximal function relative to $\{\Gamma(Q)\}_{Q \in \partial\Omega}$ is
$$N(u)(Q) = \sup_{X\in \Gamma(Q)} |u(X)|$$
at each $Q \in \partial\Omega$
We also define the following variant of the non-tangential
maximal function:
\begin{equation}\label{NTMaxVar} \widetilde{N}(u)(Q) )
=\sup_{X\in\Gamma(Q)}\left(\fint_{B_{{\delta(X)}/{2}}(X)}|u(Y)|^2\diff
Y\right)^{\frac{1}{2}}.
\end{equation}
\end{definition}

When we want to emphasize dependance of square or nontangential maximal functions on the particular cone $\Gamma_a$ we shall write $S_a(u)$ or $N_a(u)$. Similarly, if we consider cones truncated at a certain height $h$ we shall use the notation $S^h(u)$, $S^h_a(u)$, $N^h(u)$ or $N^h_a(u)$. In general, the particular choice of the aperture $a$ does not matter, as operators with different apertures give rise to comparable $L^p$ norms.

In Proposition 2.5 of \cite{DPcplx}, it was shown that, under the assumption that $|\nabla A(X)| \lesssim \delta^{-1}(X)$,
a reverse H\"older inequality for the gradient of solutions holds and therefore,  $\widetilde{N}(u)(Q)$ is comparable to $N(u)(Q)$, with possibly different apertures of
the cones used to define these quantities.

We recall the definition of $L^p$ solvability of the Dirichlet problem. 
When an operator $L$ is uniformly elliptic, the Lax-Milgram lemma can be applied and guarantees the existence of weak solutions.
That is, 
given any $f\in \dot{B}^{2,2}_{1/2}(\partial\Omega)$, the homogenous space of traces of functions in $\dot{W}^{1,2}(\Omega$), there exists a unique (up to a constant)
$u\in \dot{W}^{1,2}(\Omega$) such that $Lu=0$ in $\Omega$ and ${\rm Tr}\,u=f$ on $\partial\Omega$. These \lq\lq  energy solutions" are used to define the solvability of the $L^p$ Dirichlet, Regularity and Neumann problems.

\begin{definition} Let $1<p\le\infty$.
The Dirichlet problem with data in $L^p(\partial \Omega, d\sigma)$
is solvable (abbreviated $(D)_{p}$) if for every $f\in \dot{B}^{2,2}_{1/2}(\partial\Omega) \cap  L^p(\partial\Omega)$ the weak solution $u$ to the problem $Lu=0$ with
continuous boundary data $f$ satisfies the estimate
\[\|N(u)\|_{L^p(\partial \Omega, d\sigma)} \lesssim \|f\|_{L^p(\partial \Omega, d\sigma)}.\]
The implied constant depends only the operator $L$, $p$, and the
Lipschitz norm of $\varphi$. 
\end{definition}

\begin{definition}\label{DefRpcondition}
Let $1<p<\infty$. The regularity problem with boundary data in
$H^{1,p}(\partial\Omega)$ is solvable (abbreviated $(R)_{p}$), if
for every $f\in \dot{B}^{2,2}_{1/2}(\partial\Omega)$ with 
$\nabla_T f \in L^{p}(\partial\Omega),$
the weak solution $u$ to the problem
\begin{align*}
\begin{cases}
Lu &=0 \quad\text{ in } \Omega\\
u|_{\partial B} &= f \quad\text{ on } \partial \Omega
\end{cases}
\end{align*}
satisfies
\begin{align}
\nonumber \quad\|\widetilde{N}(\nabla u)\|_{L^p(\partial
\Omega)}\lesssim \|\nabla_T f\|_{L^{p}(\partial\Omega)}.
\end{align}
The implied constant depends only the operator $L$, $p$,
and the Lipschitz norm of $\varphi$. 
\end{definition}

\begin{definition}\label{NPDefNpcondition}
Let $1<p<\infty$. The Neumann problem with boundary data in
$L^p(\partial\Omega)$ is solvable (abbreviated $(N)_{p}$), if for
every $f\in L^p(\partial\Omega)\cap \dot{B}^{2,2}_{-1/2}(\partial\Omega)$ with the property that
$\int_{\dom} fd\sigma=0$, the weak solution $u$ to the problem
\begin{align*}
\begin{cases}
Lu &=0 \quad\text{ in } \Omega\\
A\nabla u\cdot \nu &= f \quad\text{ on }
\partial \Omega
\end{cases}
\end{align*}
satisfies
\begin{align}
\nonumber \quad\|\widetilde{N}(\nabla u)\|_{L^p(\partial
\Omega)}\lesssim \|f\|_{L^{p}(\partial\Omega)}.
\end{align}
Again, the implied constant depends only the operator $L$, $p$,
and the Lipschitz norm of $\varphi$.  Here $\nu$ is the outer
normal to the boundary $\dom$. The sense in which $A\nabla u\cdot \nu = f$ on $\partial \Omega$ 
is that 
$$\int_{\Omega} A\nabla u. \nabla \eta \,\,dX = \int_{\partial \Omega} f \eta \,d\sigma,$$
for all $\eta \in C_0^{\infty}(\mathbb R^n).$
\end{definition}

We now  compile some results from other papers that will be used to prove the main theorems.

The following is a key lemma for the square function, which combines Lemmas 3.2 and 3.3 of \cite{DPR}.

\begin{lemma}\label{l1a} Let $u$ be a solution of $Lu=0$, where $L=\div(A\nabla u)$ is a uniformly elliptic differential operator
defined on $\Omega_{t_0}$ with bounded coefficients such
that
\begin{equation}\label{carl}
|\nabla A|^2\delta(X) \qquad\mbox{is a Carleson measure, and} \qquad \delta(X) |\nabla A| < C.
\end{equation}

Then there exists $K>0$ depending only on the Lipschitz constant
of the domain $\Omega$, the Carleson norm of $|\nabla A|^2\delta(X)$, the ellipticity constant of $L$, and the dimension
$n$ such that
\begin{eqnarray}
\label{e1} &&\int_{\partial\Omega}S^2_{r/2}(\nabla u)\,d\sigma
\simeq \iint_{\dom\times (0,r/2)}|\nabla(\nabla  u(X))|^2
\delta(X)\,dX
\\\nonumber &\le &K\left[\int_{\dom}
|\nabla u|^2\,d\sigma+ \int_{\dom}N_r^2(\nabla u)
d\sigma+\frac{1}{r}\|\nabla u\|^2_{L^2(\Omega)}\right],
\end{eqnarray}
Here $N_r$ and $S_r$ denote the truncated non-tangential maximal
function and square, defined with respect to truncated cones $\Gamma_a(Q) \cap B_r(Q)$.
\end{lemma}
\section{Proof of Theorem \ref{t1}: reduction to the block case}\label{s3}

 Let $A=(a_{ij})$ be as in Theorem \ref{t1} and set $\Omega = \mathbb R^n_+$. As explained earlier, it remains to prove one of the implications
 in the statement of the Theorem. 
 To that end, assume that the $L^{q'}$ Dirichlet problem is solvable for the adjoint operator $\mathcal L^*$ for some $q>1$.

 We want to deduce solvability of the $L^q$ Regularity problem for the operator ${\mathcal L}u=\partial_i(a_{ij}\partial_j u)$ 
 if, in addition, we have solvability of the $L^q$ Regularity problem for the block-form operator
\begin{equation}\label{e0}
{\mathcal L}_0 u= \mbox{\rm div}_\parallel(A_\parallel \nabla_\parallel u)+u_{tt},
\end{equation}
where $A_{\parallel}$ is the matrix $(a_{ij})_{1\le i,j\le n-1}.$

 Throughout this section, we make the assumption that $|\nabla A(x,t)|$ is bounded by a constant $M$ for all $(x,t)$.
 All the estimates established below will be independent of $M$. This assumption entails that boundary integrals like those in
 \eqref{e8}, \eqref{e9}, and so on, are meaningful in a pointwise sense. The assumption can be removed by approximating a matrix
 that satisfies condition \eqref{Carl2} by a sequence of matrices with bounded gradients - details can be found in
 section 7 of \cite{Dsystems}.

We start by summarising useful results from \cite{KP2}. Let us denote by $\tilde{N}_{1,\varepsilon}$ the $L^1$-averaged version of the non-tangential maximal function for {\it doubly truncated} cones. That is, for $\vec{u}:\Omega\to\mathbb R^m$, we set

$$\tilde{N}_{1,\varepsilon}(\vec{u})(Q)=\sup\left\{\vint_{Z\in B(X,\delta(X)/2)}|\vec{u}|dZ:\, X\in\Gamma_\varepsilon(Q):=\Gamma(Q)\cap
\{X: \varepsilon < \delta(X)< 1/\varepsilon\}\right\}.$$

Lemma 2.8 of \cite{KP2}, stated below, provides a way to estimate the $L^q$ norm of $\tilde{N}_{1,\varepsilon}(\nabla F)(Q)$ via duality (based on tent-spaces).

\begin{lemma}\label{l1} There exists $\vec{\alpha}(X,Z)$ with $\vec{\alpha}(X,\cdot):B(X,\delta(X)/2)\to{\mathbb R}^n$ and \newline $\|\vec\alpha(X,\cdot)\|_{L^\infty(B(X,\delta(X)/2))}=1$, a nonnegative scalar function $\beta(X,Q)\in L^1(\Gamma_\varepsilon(Q))$ with $\int_{\Gamma_\varepsilon(Q)}\beta(X,Q)\,dX=1$ and a nonnegative $g\in L^{q'}(\partial\Omega,d\sigma)$ with $\|g\|_{L^{q'}}=1$ such that
\begin{equation}
\left\|\tilde{N}_{1,\varepsilon}(\nabla F)\right\|_{L^q(\partial\Omega,d\sigma)}\lesssim \int_{\Omega}\nabla F(Z)\cdot \vec{h}(Z)\, dZ,
\label{e1a}
\end{equation}
where 
$$\vec{h}(Z)=\int_{\partial\Omega}\int_{\Gamma(Q)}g(Q)\vec{\alpha}(X,Z)\beta(X,Q)\frac{\chi(2|X-Z|/\delta(X))}{\delta(X)^n}\,dX\,dQ,$$
and $\chi(s)=\chi_{(0,1)}(|s|)$.

Moreover, for any $G:\Omega\to\mathbb R$ with $\tilde{N}_{1}(\nabla G)\in L^q(\partial\Omega,d\sigma)$ we also have an upper bound
\begin{equation}
\int_{\Omega}\nabla G(Z)\cdot h(Z)\, dZ\lesssim \left\|\tilde{N}_{1}(\nabla G)\right\|_{L^q(\partial\Omega,d\sigma)}.
\label{e2}
\end{equation}
The implied constants in \eqref{e1a}-\eqref{e2} do not depend on $\varepsilon$, only on the dimension $n$.
\end{lemma}

For the matrix $A = (a_{ij})$ as above, we let $v:\Omega\to\mathbb R$ be the solution of the inhomogenous Dirichlet problem for the operator ${\mathcal L}^*$ (adjoint to ${\mathcal L}$):
\begin{equation}
{\mathcal L}^*v=\div(A^*\nabla v)=\div(\vec{h})\mbox{ in }\Omega,\qquad v\Big|_{\partial\Omega}=0.\label{e3}
\end{equation}

Then Lemma 2.10 - Lemma 2.13 of \cite{KP2} gives us the following estimates for the nontangential maximal and square functions of $v$.

\begin{lemma}\label{l2} If the $L^{q'}$ Dirichlet problem is solvable for the operator ${\mathcal L}^*$, where $q>1$, 
 then there exists $C<\infty$ depending on $n$, $q$, and $\mathcal L^*$, such that for any $\vec{h}$ as in Lemma \ref{l1} and $v$ defined by \eqref{e3} we have
\begin{equation}
\|N(v)\|_{L^{q'}(\partial\Omega,d\sigma)}+\|\tilde{N}(\delta\nabla v)\|_{L^{q'}(\partial\Omega,d\sigma)}+\|S(v)\|_{L^{q'}(\partial\Omega,d\sigma)}\le C.
\label{e4}
\end{equation}
\end{lemma}
\vglue5mm

 Let $u$ be the solution of the following boundary value problem
\begin{equation}
{\mathcal L}u=\div(A\nabla u)=0\mbox{ in }\Omega,\qquad u\Big|_{\partial\Omega}=f,\label{e5}
\end{equation}
where we assume that $f\in \dot{W}^{1,q}(\partial\Omega)\cap \dot{B}^{2,2}_{1/2}(\partial\Omega)$ for some $q>1$. Then clearly, $u\in \dot{W}^{1,2}(\Omega)$ by Lax-Milgram. 
Fix $\varepsilon>0$. Our aim is to estimate $N_{1,\varepsilon}(\nabla u)$ in $L^q$ using Lemma \ref{l1}. Let $\vec{h}$ be as in Lemma \ref{l1} for $\nabla F=\nabla u$. Then since $\vec{h}\big|_{\partial\Omega}=0$ and $\vec{h}$ vanishes at $\infty$, we have by integration by parts
\begin{equation}\label{e6}
\|N_{1,\varepsilon}(\nabla u)\|_{L^q}\lesssim \int_{\Omega}\nabla u\cdot\vec{h}\,dZ=-\int_\Omega u\,\div\vec{h}\,dZ=-\int_{\Omega}u\,{\mathcal L}^*v\,dZ=-\int_{\Omega}u\,\div(A^*\nabla v)\,dZ.
\end{equation}
We now move $u$ inside divergence and then apply the divergence theorem to obtain:
\begin{equation}\nonumber
\mbox{RHS of \eqref{e6}}=-\int_{\Omega}\div(uA^*\nabla v)\,dZ+\int_\Omega A\nabla u\cdot\nabla v\,dZ=
\int_{\partial\Omega}u(\cdot,0)a^*_{nj}\partial_jv\,dx,
\end{equation}
since
$$\int_\Omega A\nabla u\cdot\nabla v\,dZ=-\int_\Omega\mathcal Lu\,v\,dZ=0.$$
Here there is no boundary integral as $v$ vanishes on the boundary of $\Omega$. It follows that
\begin{equation}\label{e8}
\|N_{1,\varepsilon}(\nabla u)\|_{L^q}\lesssim\int_{\partial\Omega}u(x,0)a^*_{nj}(x,0)\partial_jv(x,0)\,dx,
\end{equation}
where the implied constant in \eqref{e8} is independent of $\varepsilon>0$. Now, we use the fundamental theorem of calculus and the decay of $\nabla v$ at infinity to write \eqref{e8} as
\begin{equation}\label{e9}
\|N_{1,\varepsilon}(\nabla u)\|_{L^q}\lesssim-\int_{\partial\Omega}u(x,0)\left(\int_0^\infty \frac{d}{ds}\left(a^*_{nj}(x,s)\partial_jv(x,s)\right)ds\right)dx.
\end{equation}
Recall that $\div(A^*\nabla v)=\div(\vec{h})$ and hence RHS of \eqref{e9} equals to
\begin{equation}\label{e10}
=\int_{\partial\Omega}u(x,0)\left(\int_0^\infty \left[\sum_{i<n}\partial_i(a^*_{ij}(x,s)\partial_jv(x,s))-\div\vec{h}(x,s)\right]ds\right)dx.
\end{equation}
We integrate by parts moving $\partial_i$ for $i<n$ onto $u(\cdot,0)$. The integral term containing $\partial_nh_n(x,s)$ does not need to be considered as it equals to zero by the fundamental theorem of calculus since $\vec{h}(\cdot,0)=\vec{0}$ and
$\vec{h}(\cdot,s)\to\vec{0}$ as $s\to\infty$). 

It follows that
\begin{eqnarray}\nonumber
\|N_{1,\varepsilon}(\nabla u)\|_{L^q}&\lesssim& \int_{\partial\Omega} \nabla_\parallel f(x)\cdot\left(\int_0^\infty \left[\vec{h}_\parallel(x,s)-(A^*\nabla v)_\parallel(x,s)\right]ds\right)dx\\
&=&I+II.\label{e11}
\end{eqnarray}
Here $I$ is the term containing $\vec{h}_\parallel$ and $II$ contains $(A^*\nabla v)_\parallel$. The notation we are using here  is that, for a vector $\vec{w}=(w_1,w_2,\dots,w_n)$, the vector $\vec{w}_\parallel$ denotes the first $n-1$ components of $\vec{w}$, that is $(w_1,w_2,\dots,w_{n-1})$. 

As shall see below we do not need worry about term $I$. This is because what we are going to do next is essentially undo the integration by parts we have done above but we swap function $u$ with another better behaving function $\tilde{u}$ with the same boundary data. Doing this we eventually arrive to $\|\tilde N(\nabla\tilde{u})\|_{L^q}$ plus some error terms (solid integrals) that arise from the fact that $u$ and $\tilde{u}$ disagree inside the domain. This explain why we get the same boundary integral as $I$ but with opposite sign as this \lq\lq reverse process" will undo and eliminate all such boundary terms. 

We solve a new auxiliary PDE problem to define $\tilde{u}$. Let $\tilde{u}$ be the solution
of the following boundary value problem for the operator $\mathcal L_0$ defined in \eqref{e0}, i.e., its matrix $A_0$ has the block-form
$
A_0=\left[ \begin{array}{c|c}
   A_\parallel & 0 \\
   \midrule
   0 & 1 \\
\end{array}\right] 
$ and

\begin{equation}
{\mathcal L}_0 \tilde{u}=\div(A_0\nabla\tilde{u})=0\mbox{ in }\Omega,\qquad \tilde{u}\Big|_{\partial\Omega}=f.\label{e5a}
\end{equation}
Recall that we have assumed that the $L^q$ Regularity problem for the operator 
${\mathcal L_0}$ is solvable; that is, for a constant $C>0$ independent of $f$, 
$\|\tilde{N}(\nabla \tilde{u})\|_{L^q} \le C\|\nabla_\parallel f\|_{L^q}.$ Then, by Lemma \ref{l1a}, taking $r \rightarrow \infty$, we see that
\begin{equation}
\|\tilde{N}(\nabla \tilde{u})\|_{L^q}+\|S(\nabla\tilde{u})\|_{L^q}\le C\|\nabla_\parallel f\|_{L^q}.
\label{e12}
\end{equation}

We look the term $II$. 
Let
\begin{equation}
\vec{V}(x,t)=-\int_t^\infty (A^*\nabla v)_\parallel(x,s)ds.
\end{equation}
It follows that by the fundamental theorem of calculus
$$II=\int_{\partial\Omega}\nabla_\parallel u(x,0)\cdot \vec{V}(x,0)dx=\int_{\Omega}\partial^2_{tt}\left[\nabla_\parallel \tilde{u}(x,t)\cdot \vec{V}(x,t)\right]t\,dx\,dt.$$
Hence
\begin{eqnarray}
II&=&\int_{\Omega}\partial^2_{tt}(\nabla_\parallel \tilde{u})\cdot \vec{V}(x,t)t\,dx\,dt+\int_{\Omega}\partial_{t}(\nabla_\parallel \tilde{u})\cdot \partial_t(\vec{V}(x,t))t\,dx\,dt\nonumber+\\&&+\int_{\Omega}\nabla_\parallel \tilde{u}\cdot \partial^2_{tt}(\vec{V}(x,t))t\,dx\,dt
=II_1+II_2+II_3.
\end{eqnarray}
Here $\tilde{u}$ is same as in \eqref{e5a} (observe that $u$ and $\tilde u$ have the same boundary data).
Since $ \partial_t\vec{V}(x,t)=(A^*\nabla v)_\parallel$ the term $II_2$ is easiest to handle and can be estimated as a product of two square functions
\begin{equation}\label{125}
|II_2|\le \|S(\partial_t\tilde{u})\|_{L^q}\|S(v)\|_{L^{q'}}.
\end{equation}

By our assumption that the $L^{q'}$ Dirichlet problem for the operator ${\mathcal L}^*$ is solvable, Lemma \ref{l2} applies and
provides us with an estimate $\|S(v)\|_{L^{q'}}\le C$. Combining this estimate with \eqref{e12} yields
\begin{equation}
|II_2|\le C\|\nabla_\parallel f\|_{L^q},
\end{equation}
as desired.

Next we look at $II_1$. We integrate by parts moving $\nabla\parallel$ from $\tilde{u}$. This gives us
\begin{equation}\label{eq335}
II_1 =\int_{\Omega}\partial^2_{tt}\tilde{u}\cdot\left(\int_{t}^\infty\div_{\parallel}(A^*\nabla v)_{\parallel}ds\right)t\,dx\,dt.
\end{equation}
Using the PDE $v$ satisfies we get that
$$\int_{t}^\infty\div_{\parallel}(A^*\nabla v)_{\parallel}ds=(a_{nj}\partial_jv)(x,t)+\int_{t}^\infty\div\vec{h}\,ds.$$
Using this in \eqref{eq335} we see that
\begin{equation}\label{eq335a}
II_1 =\int_{\Omega}(\partial^2_{tt}\tilde{u})(a_{nj}\partial_jv)t\,dx\,dt+\int_{\Omega}\partial^2_{tt}\tilde{u}
\cdot\left(\int_{t}^\infty\div\vec{h}\,ds\right)t\,dx\,dt.
\end{equation}
Here the first term enjoys the same estimate as $II_2$, namely \eqref{125}. We work more with the second term which we call $II_{12}$. We integrate by parts in $\partial_t$. 
\begin{equation}\label{eq335b}
II_{12} =\int_{\Omega}(\partial_{t}\tilde{u})(
\div\vec{h})t\,dx\,dt-\int_{\Omega}\partial_{t}\tilde{u}
\cdot\left(\int_{t}^\infty\div\vec{h}\,ds\right)\,dx\,dt=
\end{equation}
\begin{equation}\nonumber
=II_{121}-\int_{\Omega}\partial_{t}\tilde{u}
\cdot\left(\int_{t}^\infty\div\vec{h}\,ds\right)\,dx\,dt=II_{121}+\int_{\partial\Omega}\tilde{u}(x,0)\left(\int_0^\infty \div\vec{h}\right)dx+\int_{\Omega}\nabla\tilde{u}\cdot\vec{h}\,dx\,dt
\end{equation}
\begin{equation}\nonumber
=II_{121}-\int_{\partial\Omega}\nabla_\parallel \tilde{u}(x,0)\left(\int_0^\infty \vec{h}_\parallel\right)dx+II_{123}=
II_{121}-I+II_{123}.\hskip4.3cm
\end{equation}
In the second line we have swapped $\partial_t$ and $\partial_{\parallel}$ derivatives  integrating by parts twice. 
This integration yields a boundary term but fortunately this term is precisely as the term $I$ defined by \eqref{e11} but since it comes with opposite sign these two terms cancel out.  We return to the terms $II_{121}$ and $II_{123}$ later.

Next we look at $II_3$. We see that
\begin{equation}
II_3= \int_{\Omega}\nabla_\parallel \tilde{u}\cdot \partial_t(A^*\nabla v)_\parallel t\,dx\,dt=
\int_{\Omega}\nabla_\parallel \tilde{u}\cdot ((\partial_t A)^*\nabla v)_\parallel t\,dx\,dt+\int_{\Omega}\nabla_\parallel \tilde{u}\cdot (A^*\nabla (\partial_t v))_\parallel t\,dx\,dt
\end{equation}
\begin{equation}
=II_{31}+II_{32}.\label{qq}
\end{equation}

In order to handle the term $II_{31}$ we will use the fact that the matrix A satisfies the Carleson measure condition \eqref{Carl2}.
The argument uses a stopping time argument that is typical in connection with Carleson measures.

To set this up, let $\mathcal O_j$ denote $\{x \in \partial\Omega:  N(\nabla\tilde{u})(Q)S(v)(Q) > 2^j\}$ and 
define an enlargement of $\mathcal O_j$ by $\tilde{\mathcal O}_j := \{ M(\chi_{\mathcal O_j}) > 1/2\}$. (Note that  $|\tilde{\mathcal O}_j | \lesssim |\mathcal O_j|$.)
We will break up integrals over $\Omega$ into regions determined by the sets:
$$F_j = \{X = (y,t) \in \Omega: |\Delta_{ct}(y) \cap \mathcal O_j| > 1/2, \,\,|\Delta_{ct}(y) \cap \mathcal O_{j+1}| \leq 1/2\},$$
where $c$ depends on the aperture of the cones used to define the nontangential maximal function and square functions.

Then, 

\begin{eqnarray}
|II_{31}|&\lesssim& \int_{\Omega} |\nabla\tilde{u}||\partial_t A||\nabla v|t dX \le \sum_j \int_{\Omega \cap F_j} |\nabla\tilde{u}||\partial_t A||\nabla v|t dX \nonumber\\
&\le& \sum_j \int_{\tilde{\mathcal O}_j \setminus \mathcal O_j} \int_{\Gamma(Q) \cap F_j} |\nabla\tilde{u}||\partial_t A||\nabla v|t^{2-n} dX dx\nonumber\\
&\le&\sum_j \int_{\tilde{\mathcal O}_j \setminus \mathcal O_j} \left(\int_{\Gamma(Q)}|\nabla v|^2 |\nabla \tilde{u}|^2 t^{2-n}dX\right)^{1/2}\left(\int_
{\Gamma(Q) \cap F_j}|\partial_tA|^2|^2t^{2-n}dX\right)^{1/2}dQ\nonumber
\end{eqnarray}
\begin{eqnarray}
&\le&\sum_j \int_{\tilde{\mathcal O}_j \setminus \mathcal O_j} N(\nabla\tilde{u})(Q) S(v)(Q) \left(\int_
{\Gamma(Q) \cap F_j}|\partial_tA|^2|^2t^{2-n}dX\right)^{1/2} \,dQ\nonumber\\
&\le&\sum_j 2^j \left(\int_{\tilde{\mathcal O}_j}  \int_
{\Gamma(Q) \cap F_j}|\partial_tA|^2|^2t^{2-n}dX\,dQ\right)^{1/2}  |\tilde{\mathcal O}_j|^{1/2}   \nonumber\\
&\lesssim&\sum_j 2^j  |\mathcal O_j| \,\,\lesssim\,\, \int_{\partial\Omega} N(\nabla\tilde{u})(Q) S(v)(Q) dQ.\label{Stopping}
\end{eqnarray}

 The penultimate inequality follows from the Carleson measure property of $|\partial_tA|^2| t dX$ as the integration is over the
 Carleson region $\{X=(y,t): \Delta_{ct}(y) \subset \tilde{\mathcal O}_j\}$.
 
 Consequently, by H\"older's inequality,
\begin{equation}
|II_{31}|\lesssim \|S(v)\|_{L^{q'}}\|N(\nabla\tilde{u})\|_{L^q}.
\end{equation}
Hence as above 
\begin{equation}
|II_{31}|\le C\|\nabla_\parallel f\|_{L^q}.
\end{equation}
For the term $ II_{32}$ we separate the parallel and tangential parts of the gradient, to get 
\begin{eqnarray}\nonumber
II_{32}&=&\int_{\Omega}\nabla_\parallel \tilde{u}\cdot (A_\parallel^*\nabla_\parallel (\partial_t v))t\,dx\,dt+
\int_{\Omega}\nabla_\parallel \tilde{u}\cdot (a_{in}^*\partial^2_{tt} v)_{i<n}t\,dx\,dt\\
&=&-\int_{\Omega}\div_\parallel(A_\parallel\nabla \tilde{u})(\partial_tv)tdx\,dt+II_{33}=\int_{\Omega}(\partial^2_{tt}\tilde{u})(\partial_tv)tdx\,dt+II_{33}.\nonumber
\end{eqnarray}
Here we have integrated the first term by parts and then used the equation that $\tilde{u}$ satisfies. It follows that in the last expression the first term has square functions bounds identical to \eqref{125}. For $II_{33}$
we write $\partial^2_{tt}v$ as
$$\partial^2_{tt}v=\partial_t\left(\frac{a^*_{nn}}{a^*_{nn}}\partial_t v\right)=\frac1{a^*_{nn}}\partial_t(a^*_{nn}\partial_tv)-\frac{\partial_t a^*_{nn}}{a^*_{nn}}\partial_t v$$
$$=-\frac1{a^*_{nn}}\left[\div_\parallel(a^*_\parallel\nabla_\parallel v)+\sum_{i<n}\left[\partial_i(a^*_{in}\partial_t v)+\partial_t(a^*_{ni}\partial_i v)\right]+\partial_t(a^*_{nn})\partial_t v -\div\vec{h}\right],$$
where the final line follows from the equation that $v$ satisfies. It therefore follows that the term $II_{33}$ can be written as a sum of five terms (which we shall call $II_{331},II_{332},\dots,II_{335}$.)

Terms $II_{331}$ and $II_{332}$ are similar and we deal with then via integration by parts (in $\partial_i$, $i<n$):
\begin{equation}\label{eq133}
|II_{331}|+|II_{332}|\le C\int_\Omega|\nabla^2 \tilde{u}||\nabla v|t+C\int_\Omega|\nabla A||\nabla\tilde{u}||\nabla v|t.
\end{equation}
For the third term $II_{333}$ we observe that $\partial_t(a^*_{ni}\partial_i v)=\partial_i(a^*_{ni}\partial_t v)+(\partial_ta^*_{ni})\partial_iv-(\partial_ia^*_{ni})\partial_tv$ which implies that it again can be estimated by the right-hand side of \eqref{eq133}. 
The same is true for the term $II_{334}$ which has a bound by the second term on the right-hand side of \eqref{eq133}. It remains to consider the term  $II_{335}$ which is
\begin{equation}\label{eq134}
II_{335}=\sum_{i<n}\int_{\Omega}\frac{a_{ni}}{a_{nn}}\partial_i\tilde{u}\,(\div\vec{h})\,t\,dx\,dt.
\end{equation}
Notice the similarity of this term with $II_{121}$, hence the calculation below also applies to it.
We again integrate by parts. Observe we get an extra term when $\partial_t$ derivative falls on $t$. This gives us
\begin{equation}\label{eq135}
|II_{121}|+|II_{335}|\le C\int_\Omega|\nabla^2 \tilde{u}||\vec{h}|t+\int_\Omega|\nabla A||\nabla\tilde{u}||\vec{h}|t+\sum_{i}\left|\int_\Omega \frac{a_{ni}}{a_{nn}}\partial_i\tilde{u}\,h_n\,\,dx\,dt\right|.
\end{equation}

We deal with terms on the right-hand side of \eqref{eq133} and \eqref{eq135} now. The first term of \eqref{eq133}  can be seen to be a product of two square functions and hence by H\"older it has an estimate by $\|S(\nabla\tilde{u})\|_{L^q}\|S(v)\|_{L^{q'}}$. The second term of \eqref{eq133} is similar to the term $II_{31}$ with analogous estimate. It follows that
\begin{eqnarray}\nonumber
|II_{331}|+|II_{332}|+|II_{333}|+|II_{334}|&\le& C(\|S(\nabla\tilde{u})\|_{L^q}+\|N(\nabla\tilde{u})\|_{L^q})\,\|S(v)\|_{L^{q'}}\\\label{eq136}
&\le& C\|\nabla_\parallel f\|_{L^q},
\end{eqnarray}
by using \eqref{e12} and Lemma \ref{l2}. The first two terms of \eqref{eq135} have similar estimates, provided we introduce as in \cite{KP2} the operator $\tilde{T}$. Here
$$\tilde{T}(|\vec{h}|)(Q)=\int_{\Gamma(Q)}|\vec{h}|(Z)\delta(Z)^{1-n}(Z)dZ.$$
The last term of \eqref{eq135} and also the term $II_{123}$ is handled using \eqref{e2}. Here the presence of $\frac{a_{ni}}{a_{nn}}$ in the integral is harmless as we have flexibility to hide this term into the vector-valued function $\vec\alpha$ in the definition of $\vec{h}$. This gives us
\begin{eqnarray}\nonumber
|II_{121}|+|II_{123}|+|II_{335}|&\le& C(\|S(\nabla\tilde{u})\|_{L^q}+\|N(\nabla\tilde{u})\|_{L^q})\,\|\tilde{T}(|\vec{h}|)\|_{L^{q'}}+C\|\tilde{N}_1(\nabla\tilde{u})\|_{L^q}\\\label{eq136x}
&\le& C\|\nabla_\parallel f\|_{L^q}.
\end{eqnarray}
Here the bound for $\|\tilde{T}(|\vec{h}|)\|_{L^{q'}}$ follows from Lemma 2.13 of \cite{KP2}.
\vglue2mm

In summary, under the assumptions we have made we see that
$$II=\int_{\partial\Omega}\nabla_\parallel u(x,0)\cdot \vec{V}(x,0)dx\le C\|\nabla_\parallel f\|_{L^q}-I.$$

After putting all estimates together (since term $I$ cancels out), we have established the following:
$$ \|\tilde{N}_{1,\varepsilon}(\nabla u)\|_{L^q}\le C\|\nabla_\parallel f\|_{L^q}.$$

\medskip

\noindent {\bf Remark:} The assumption that $L^p$ Regularity problem for the block form operator ${\mathcal L}_\parallel$ is solvable for some $p>1$ implies solvability of the said Regularity problem for all values of $p\in (1,\infty)$. This follows by combining results of \cite{DK} and \cite{DPP}.\vglue2mm

An argument is required to demonstrate that the control of $\tilde{N}_{1,\varepsilon}(\nabla u)$ of a solution $\mathcal Lu=0$ implies the control of $\tilde{N}(\nabla u)$ (the $L^2$ averaged version of the non-tangential maximal function). Firstly, as the established estimates are independent of $\varepsilon>0$ we obtain
$$\|\tilde{N_1}(\nabla u)\|_{L^q}=\lim_{\varepsilon\to 0+}\|\tilde{N}_{1,\varepsilon}(\nabla u)\|_{L^q}\le C\|\nabla_\parallel f\|_{L^q}.$$
Secondly, as $\nabla u$ satisfies a reverse H\"older self-improvement inequality
$$\left(\vint_B|\nabla u|^{2+\delta}\right)^{1/(2+\delta)}\lesssim \left(\vint_{2B}|\nabla u|^{2}\right)^{1/2},$$
for some $\delta>0$ depending on ellipticity constant and all $B$ such that $3B\subset\Omega$, it also follows (c.f. \cite[Theorem 2.4]{S}) that
$$\left(\vint_B|\nabla u|^2\right)^{1/2}\lesssim \left(\vint_{2B}|\nabla u|\right),$$
which implies a bound of $\tilde{N}(\nabla u)(\cdot)$ defined using cones $\Gamma_a(\cdot)$ of some aperture $a>0$ by $\tilde{N}_1(\nabla u)(\cdot)$ defined using cones $\Gamma_b(\cdot)$ of some slightly larger aperture $b>a$. Hence $ \|\tilde{N}(\nabla u)\|_{L^q}\le C\|\nabla_\parallel f\|_{L^q}$ must hold.

This completes the proof of Theorem \ref{t1}. \qed

\section{The Regularity and Neumann problems when $n=2$}

To prove Theorem \ref{RNduality} in the special case $n=2$, we will use Theorem \ref{t1}, a change of variable discovered by Feneuil \cite{F}, and the equivalence in dimension two between the solvability of Regularity and Neumann problems observed by Kenig and Rule \cite{KR}. 

\bigskip

\noindent{\it Proof of Theorem \ref{RNduality}.} The solvability of the Neumann problem
can be reduced to solvability of the Regularity problem using an observation in \cite{KR}; namely, if $u$
solves $\mathcal Lu=\div(A\nabla u)=0$ in $\Omega$ then $\tilde{u}$ uniquely (modulo constants) defined via
\begin{equation}
\begin{bmatrix}
0 & -1\\ 1 & 0
\end{bmatrix}\nabla{\tilde{u}}=A\nabla u
\end{equation}
solves the equation $\tilde{\mathcal L}u=\div(\tilde{A}\nabla u)=0$ with $\tilde{A}=A^t/\det{A}$ and the tangential derivative of $u$ is the co-normal derivative of $\tilde{u}$ and vice-versa. 

If $A$ satisfies the Carleson condition \eqref{CC} then so does $A^t/\det{A}$ (with a possibly larger constant) and hence the $L^p$ Neumann problem for a given matrix $A$ is solvable in the same range $1<p<p_{max}$ for which the $L^p$ Regularity problem for the matrix $A^t/\det{A}$ is solvable. The range of solvability for the operator with matrix $A^t/\det{A}$ is determined by the range of solvability of 
the Dirichlet problem for its adjoint operator, which has matrix $A/\det{A}$, reducing the second claim of Theorem \ref{RNduality} to the first.
In summary, for these operators, the solvability of the Neumann problem can be deduced from solvability of the Regularity problem.\vglue2mm

Next, we perform some additional reductions to simplify the problem. A well known and used pull-back transformation
\begin{equation}\label{CHV}
(x,t)\mapsto (x,ct+(\eta_t*\phi)(x))
\end{equation}
for a smooth family of mollifiers $(\eta_t)_{t>0}$ and some sufficiently large enough $c>0$ (depending on $\|\nabla\phi\|_{L^\infty}$) allows us to consider the Regularity problem on the domain ${\mathbb R}^n_+$. This is because, for $\Omega=\{(x,t):\, t>\phi(x)\}$, the pull-back map preserves the ellipticity condition and the Carleson condition on the coefficients (although the Carleson bound $K$ for the new operator on ${\mathbb R}^n_+$ might increase and will depend on $\|\nabla\phi\|_{L^\infty}$ as well).

Hence from now on we assume that $\Omega={\mathbb R}^n_+$. The next reduction comes in the form of replacing the Carleson condition \eqref{CC} by the stronger condition: 
\begin{equation}\label{CCalt}
\delta(X)\left[\sup_{Y\in B(X,\delta(X)/2)}|\nabla A(Y)|\right]^2\mbox{ is a Carleson measure}.
\end{equation}
To see this, one consider a new matrix $\bar{A}$ obtained from $A$ via mollification
$\bar{A}(x,t)=(A*\eta_{t/2})(x,t)$ (for details see \cite{DPR} where this observation was made). The matrix valued function $\bar{A}$ is uniformly elliptic but now satisfies \eqref{CCalt} instead of the oscillation condition, \eqref{CC}, that holds for $A$. In addition, we also have
\begin{equation}\label{CCpert}
\delta(X)^{-1}\left[\sup_{Y\in B(X,\delta(X)/2)}|A(Y)-\bar{A}(Y)|\right]^2\mbox{ is a Carleson measure.}
\end{equation}

Let us clarify our objective. It suffices to prove that the $L^p$ Regularity problem for the original operator $\mathcal L=\div(A\nabla\cdot)$ is solvable for at least one value $q\in(1,\infty)$ as then by \cite{DK} it follows that $L^p$ Regularity problem for $\mathcal L$ is solvable if and only if the $L^{p'}$ Dirichlet problem is solvable for the adjoint operator $\mathcal L^*$. (See Theorem 1.1 of \cite{DK}) But solvability of Dirichlet problem satisfying Carleson condition in the range $(p_{dir},\infty)$ for some $p_{dir}>1$ has been resolved in \cite{KP}, and hence the claim about the range of solvability of Regularity stated in Theorem  \ref{RNduality} would follow.

The operator with matrix $A$ is, by \eqref{CCpert}, a Carleson perturbation of the operator $\bar{A}$. By the perturbation 
theory of \cite{KP2}, the solvability of the $L^q$ Regularity problem for at least one $q\in (1,\infty)$ for the operator $\div(\bar{A}\nabla\cdot)$ 
implies solvability of the Regularity problem for the operator 
$\mathcal L=\div(A\nabla\cdot)$ for a possibly different (smaller) value of $\tilde{q}>1$. \vglue2mm

Hence it remains to establish the solvability of the  $L^q$ Regularity problem for a uniformly elliptic operator satisfying condition \eqref{CCalt} in the domain $\Omega=\mathbb R^n_+$ for at least one value of $q>1$. Up to this point, all the statements and reductions regarding 
the Regularity problem are valid in any dimension. In what follows, we will use the assumption that we are in two dimensions. 

In the two dimensional setting, we observe that if our matrix $A$ has the special form in which the $a_{11}$ coefficient equals $1$ in $\Omega$, then the operator $\mathcal L_0$ in Theorem \ref{t1} is simply the Laplacian $\Delta$. For the Laplacian, all the required square and non-tangential estimates are known, including solvability of the Regularity problem for all values of $p>1$. Applying Theorem \ref{t1} gives the solvability of the $L^q$ Regularity problem for at least one value $q>1$ for an operator $\mathcal L=\div(A\nabla\cdot)$ with $a_{11}=1$, which is our objective.
Hence the objective now is to reduce from our general matrix $A$ to one having this special form with $1$ in the top left corner.
\vglue2mm

The strategy of using a change of variables to reduce to matrices of a special form has been used before in two dimensions
to prove solvability of Dirichlet and Regularity boundary value problems. The paper \cite{KKPT2} considered the $L^p$ Dirichlet problem for operators whose matrix
$A(x,t) = A(x)$ is $t$-independent and non-symmetric. The crucial observation that was used to
resolve the obstacles in solving this problem for some, possibly large, value of $p$ was the 
discovery of a change of variables reducing matters to matrices of the form $\begin{bmatrix} 1 &\gamma(x)\\0& \delta(x)\end{bmatrix}$. For these
matrices, the proof of solvability used arguments that took advantage of the upper triangular structure. This particular
change of variable does not apply to the situation in Theorem \ref{RNduality} as it relied heavily on the $t$-independence.

We will be able to make the reduction to a matrix with $1$ in the top left corner via a very useful change 
of variables introduced by J. Feneuil in \cite{F}. The change of variable can be stated in $n$-dimensions, and
has strong consequences when specialized two dimensions. In ${\mathbb R}^n_+$, let:
$$\rho:(x,t)\mapsto (x,th(x,t)),$$
where $1 < h < 2$, and which, for $2t|\nabla h|<h$, is a bijection on ${\mathbb R}^n_+$.
Let $J_\rho$ denote the Jacobian of this change of variables:

If $u$ is the solution to an operator $\mathcal L=\div(A\nabla\cdot)$, then, as observed in \cite{F}, then $u \circ \rho$ is the solution
 a new operator $\mathcal L_\rho$ with matrix $A_\rho = \text{det\,}(J_\rho) (J_\rho)^{-t} (A \circ \rho) (J_\rho)^{-1}$.
 A simple calculation gives that  $J_\rho =  \left[ \begin{array}{c|c}
   I  & t\nabla_x h \\
   \midrule
   0 & h + t\partial_t h \\
\end{array}\right].$
 \vglue2mm
 
As observed in \cite{F}, the matrix $A_\rho$ can be written in the form

\begin{equation}
A_\rho=\left[ \begin{array}{c|c}
   hA_\parallel & B \\
   \midrule
   C & h^{-1}d \\
\end{array}\right]+{\mathcal B}_\rho,\quad
\end{equation}
where

\begin{equation}
A=\left[ \begin{array}{c|c}
   A_\parallel & B \\
   \midrule
   C & d \\
\end{array}\right],\mbox{ and $A_\parallel$ is the $(n-1)\times (n-1)$ block.}
\end{equation}
The matrix ${\mathcal B}_\rho$ is a Carleson perturbation matrix, that is
\begin{equation}\label{CCpert2}
\delta(X)^{-1}\left[\sup_{Y\in B(X,\delta(X)/2)}|{\mathcal B}_\rho(Y)|\right]^2\mbox{ is a Carleson measure.}
\end{equation}
 \vglue2mm
 
In particular, the result of \cite{KP2} tell us that when \eqref{CCpert2}  holds, the solvability of Regularity problem for the operator with matrix 
$\left[ \begin{array}{c|c}
   hA_\parallel & B \\
   \midrule
   C & h^{-1}d \\
\end{array}\right]$ for some $q>1$ implies solvability of Regularity problem for the operator with matrix $A_\rho$ for possibly different value of $\tilde q>1$. As $A_\rho$ and $A$ are related via the change of variables this also implies solvability of the original Regularity problem for operator with matrix $A$ for the same $\tilde{q}>1$.
 \vglue2mm
 
Consider the case $n=2$. Ideally, we would want to choose $h=a_{11}^{-1}$ so that the matrix 
$\left[ \begin{array}{c|c}
   hA_\parallel & B \\
   \midrule
   C & h^{-1}d \\
\end{array}\right]$ has a $1$ in the top left corner. It is only possible to make this choice when $2t|\nabla h|<h$. If that is not the case, 
we use the clever method of \cite{F} to achieve this objective after a finite sequence of steps instead of just one. Observe that since $a_{11}$ satisfies the Carleson condition \eqref{CCalt} and $\lambda\le a_{11}\le\Lambda$ there exists an integer $N$ such that, for $h=a_{11}^{-1/N}$, we have $2t|\nabla h|<h$ as well as the property that $1/2\le h\le 2$.
 \vglue2mm
As in \cite{F}, this can be iterated $N$ times. After one iteration, we have that the solvability of Regularity problem for an operator with matrix $A$ for some $q>1$ can be deduced from solvability of Regularity problem for an operator with matrix $\begin{bmatrix} ha_{11} & a_{12}\\ a_{22} & h^{-1}a_{22}\end{bmatrix}$. In the next iteration we relate it to the solvability of the Regularity problem for the operator with matrix $\begin{bmatrix} h^2a_{11} & a_{12}\\ a_{22} & h^{-2}a_{22}\end{bmatrix}$. Finally, after the $N$ iterations we find that we need to consider solvability of the Regularity problem for the operator with matrix $\begin{bmatrix} h^Na_{11} & a_{12}\\ a_{22} & h^{-N}a_{22}\end{bmatrix}=\begin{bmatrix} 1 & a_{12}\\ a_{22} & a_{11}a_{22}\end{bmatrix}$.  As observed above for such matrices, Theorem \ref{t1} gives solvability of the Regularity problem for some $q>1$.
 \vglue2mm

It is important to emphasise that Feneuil's change of variables  gives \eqref{CCpert2} only if coefficients of the original matrix $A$ satisfy \eqref{CCalt}, but we
have reduced matters to this situation. 
 This finishes the proof of Theorem \ref{RNduality}.
 \qed
\section{Regularity problem for block form operators when $n\ge 2$.}

In this section we establish Theorem \ref{RNduality} in all dimensions for the Regularity problem. The argument is not as simple as in the case $n=2$, where we made use of Feneuil's change of variables. Instead, the ideas necessary for the $n$-dimensional result are closer to the methods of \cite{DPR}.

Using the same reductions established in Section 4 - the flattening of the domain and mollification of coefficients - we see that Theorem \ref{RNduality} holds provided we can solve the $L^q$ Regularity problem in $\mathbb R^n_+$ for the block form operator \eqref{e0a} under the Carleson condition \eqref{CCalt} for at least one (and hence for all) $1<q<\infty$. This is precisely the claim of Theorem \ref{tblock}, and we turn to its proof.

Consider therefore $A_\parallel$ as in Theorem \ref{tblock}  and denote by $\mathcal L_0$ the operator
\begin{equation}\label{e0ax}
{\mathcal L}_0 u= \mbox{\rm div}_\parallel(A_\parallel \nabla_\parallel u)+u_{tt}.
\end{equation}

For each $k=2,3,4,\dots$ let $\mathcal L_k$ be a related rescaled operator in $t$-variable defined as follows:
\begin{equation}\label{e0axs}
{\mathcal L}_k u= \mbox{\rm div}_\parallel(A^k_\parallel \nabla_\parallel u)+u_{tt},
\end{equation}
where 
\begin{equation}\label{ss1}
A^k_\parallel(x,t)=A_\parallel(x,kt),\qquad\mbox{for all }x\in\mathbb R^{n-1}\mbox{ and }t>0. 
\end{equation}

We claim that for each $k=2,3,\dots$ the $L^q$ Regularity problem for $\mathcal L_0$ in $\mathbb R^n_+$  is solvable if and only if the $L^q$ Regularity problem for $\mathcal L_k$ in $\mathbb R^n_+$  is solvable.

This can be see as follows. Using the mean value theorem the coefficients $A^k_\parallel$ can be viewed as Carleson perturbations of coefficients of $\mathcal L_0$ which are $A_\parallel$. That is, similar to \eqref{CCpert}, we have that
\begin{equation}\label{CCpertX}
\delta(X)^{-1}\left[\sup_{Y\in B(X,\delta(X)/2)}|A_\parallel(Y)-{A^k_\parallel}(Y)|\right]^2\mbox{ is a Carleson measure.}
\end{equation}
Thus, if the $L^q$ Regularity problem for $\mathcal L_0$ in $\mathbb R^n_+$ is solvable, then so is the 
$L^{\tilde{q}}$ Regularity problem for $\mathcal L_k$ in $\mathbb R^n_+$ for some $\tilde{q}>1$ (by \cite{KP2}).
But for these block form operators, solvability of the Regularity problem for one value $\tilde{q}>1$ implies solvability for all values. Therefore we can deduce that the $L^q$ Regularity problem for $\mathcal L_k$ in $\mathbb R^n_+$ is solvable. The reverse implication has a similar proof. 

Next, we consider what we can say about the Carleson condition for the coefficients $A^k_\parallel$. We want to look at 
$$d\mu^k(x,t)=|\nabla_x A^k_\parallel (x,t)|^2t\,dx\,dt.$$
Notice that the gradient is only taken in $x$ variable, not in $t$, so we are not examining the same (full) Carleson measure property of the coefficients.
Given that \eqref{Carl2m} holds,  it follows that for 
$$d\mu^0(x,t)=|\nabla_x A_\parallel (x,t)|^2t\,dx\,dt,$$
we have that
\begin{equation}\label{ss2}
\|\mu^0\|_{Carl}\le \|\mu\|_{Carl}\qquad\mbox{and }\qquad |\nabla_x A_\parallel (x,t)|\le \frac{\|\mu\|^{1/2}_{Carl}}t.
\end{equation}
Let $\Delta\subset\mathbb R^{n-1}$ be a boundary ball of radius $r$. Let $T(\Delta)$ be the usual Carleson region associated with $\Delta$.

 To estimate the Carleson norm of $\mu^k$ in the region $T(\Delta)\cap \{X:\delta(X)<r/k\}$, a change of variables $(x,t)\mapsto (x,kt)$  together with the first
  the Carleson norm property in
  \eqref{ss2} gives an upper bound of $1/k^2$.
 In the region $T(\Delta)\cap \{X:\delta(X)\ge r/k\}$, 
we use the second estimate in \eqref{ss2} and altogether this gives:
\begin{equation}\label{ss3}
\|\mu^k\|_{Carl}\le \|\mu\|_{Carl}\frac{1+C(n)\log k}{k^2},\qquad\mbox{for some }C(n)>0.
\end{equation}
It follows that by choosing $k$ large enough we can make the Carleson norm of $\mu^k$ as small as we wish.
This observation will be crucial for what follows.

From now on let $B_\parallel=A^k_\parallel$ for some large fixed $k$ which will be determined later. Let 
\begin{equation}\label{e0axx}
{\mathcal L} u= \mbox{\rm div}_\parallel(B_\parallel \nabla_\parallel u)+u_{tt},
\end{equation}
and we consider the Regularity problem for this operator on $\Omega=\mathbb R^n_+$. Our objective now is to solve the $L^q$ Regularity problem for $\mathcal L$ for some $q>1$, thus proving  Theorem \ref{tblock}.

Suppose that $\mathcal Lu=0$ and that $u\big|_{\partial\Omega}=f$ for some $f$ with $\nabla_x f\in L^q$.

In the spirit of the approach taken in \cite{DPR} we consider the PDEs satisfied by each $w_m=\partial_m u$ for $m=1,2,
\dots, n-1$ satisfies. Due to the block form nature of our operator $\mathcal L$ we have the following:
\begin{eqnarray}\label{system}
{\mathcal L}w_m&=&\sum_{i,j=1}^{n-1}\partial_i((\partial_m b_{ij})w_j)\quad\mbox{in }\Omega,\quad m=1,2,\dots,n-1,\\
w_m\Big|_{\partial\Omega}&=&\partial_m f.\nonumber
\end{eqnarray}
Observe that only $w_1,\dots, w_{n-1}$ appears in these equations and hence \eqref{system} is a {\it weakly coupled} fully determined system of $n-1$ equations for the unknown vector valued function $W=(w_1,w_2\dots,w_{n-1})$
with boundary datum $W\big|_{\partial\Omega}=\nabla_xf\in L^p$. We call this system {\it weakly coupled} because each $\partial_mb_{ij}$ appearing on the righthand side has small Carleson measure norm, which follows from \eqref{ss3} since $k$ will be chosen to be (sufficiently) large. 

Hence, let us write $w_m=v_m+\eta_m$ where each $v_m$ solves the Dirichlet problem 
$$\mathcal Lv_m=0\mbox{ in }\Omega,\qquad v_m\big|_{\partial\Omega}=\partial_mf\in L^q(\partial\Omega).$$
As $\mathcal L$ is a block form matrix we know this Dirichlet problem is solvable for all $1<q<\infty$ and we have the following square and nontangential estimates:
\begin{equation}\label{sss1}
\|S(v_m)\|_{L^q}\approx\|N(v_m)\|_{L^q}\lesssim \|\partial_m f\|_{L^q},\qquad m=1,2,\dots,n-1.
\end{equation}
Thus each $\eta_m$ solves
\begin{eqnarray}\label{system2}
{\mathcal L}\eta_m&=&\sum_{i,j=1}^{n-1}\partial_i((\partial_m b_{ij})(v_j+\eta_j))\quad\mbox{in }\Omega,\quad m=1,2,\dots,n-1,\\
\eta_m\Big|_{\partial\Omega}&=&0.\nonumber
\end{eqnarray}
Our aim is to establish square and nontangential estimates for each $\eta_m$ as well, and thus also for $w_m$.

Let us start with the square function bound. The most convenient bound we can get is when $q=2$ and hence from now on we shall assume that. Using ellipticity we see that
\begin{eqnarray}\label{system3}
\|S(\eta_m)\|_{L^2}^2&\approx&\int_{\mathbb R^n_+}\left(\sum_{i,j=1}^{n-1}b_{ij}\partial_j\eta_m\partial_i\eta_m+(\partial_t\eta_m)^2\right)t\,dx\,dt=\\\nonumber
&=&-\int_{\mathbb R^n_+}(\mathcal L\eta_m)\eta_mt\,dx\,dt+\frac12\int_{\mathbb R^{n}_+}\partial_t(\eta_m)^2\,dx\,dt,
\end{eqnarray}
where in the second line we have integrated by parts. There is no boundary integral due to the fact that $t=0$ at the boundary. The 
second summand following the equality in \eqref{system3} vanishes since $\eta_m=0$ at the boundary and $\eta_m\to 0$ as $t\to\infty$ (due to the decay of our solutions at infinity). Hence, only the penultimate term of \eqref{system3} remains, where we will substitute \eqref{system2} and sum over $m$.
\begin{eqnarray}\nonumber
\sum_{m<n}\|S(\eta_m)\|_{L^2}^2&\approx&-\int_{\mathbb R^n_+}\sum_{i,j,m<n}\partial_i((\partial_m b_{ij})(v_j+\eta_j))\eta_m t \,dx\,dt\\\label{system4}
&=&\int_{\mathbb R^n_+}\sum_{i,j,m<n}(\partial_m b_{ij})(v_j+\eta_j)\partial_i\eta_m t \,dx\,dt\\\nonumber
&\lesssim& \left(\sum_{m<n}\|S(\eta_m)\|_{L^2}^2\right)^{1/2}\left(\int_{\mathbb R^n_+}|\nabla_x B_{\parallel}|^2(|V|^2+|\vec\eta|^2)t\,dx\,dt\right)^{1/2},
\end{eqnarray}
by Cauchy-Schwarz. Hence $V=(v_1,v_2,\dots,v_{n-1})$ and $\vec\eta=(\eta_1,\eta_2,\dots,\eta_{n-1})$. For the last term of expression above we use the Carleson property and also move the square function term on the righthand side. This will give us:
\begin{eqnarray}\label{system5}
\sum_{m<n}\|S(\eta_m)\|_{L^2}^2&\lesssim&\||\nabla_xB_\parallel|^2t\,dx\,dt\|_{Carl}\left(\sum_{m<n}
\|N(v_m)\|_{L^2}^2+\|N(\eta_m)\|_{L^2}^2\right).
\end{eqnarray}
But for nontangential maximal function of $v_m$ we do have \eqref{sss1} and hence we can conclude that
\begin{eqnarray}\label{system6}
\sum_{m<n}\|S(\eta_m)\|_{L^2}^2&\le& C(k)\left(\|\nabla_xf\|^2_{L^2}+\sum_{m<n}
\|N(\eta_m)\|_{L^2}^2\right).
\end{eqnarray}
Here $C(k)\to 0$ as $k\to\infty$ thanks to the choice of matrix $B_{\parallel}$ made above.\vglue1mm

It remain to establish a nontangetial estimates of $N(\eta_m)$ since we would like to move such terms from the righhand side of \eqref{system6}.

Here we refer the reader to the paper \cite{Dsystems} where a classical stopping time technique has been used for similar estimates in the case of systems.
(The idea of estimating an integral of a nontangential maximal via good-$\lambda$ inequalities and a Lipschitz graph determined by the stopping time goes back to
\cite{KKPT2}. New ideas were needed to make this approach work in the case of systems.)
In particular, Lemma 5.1, Lemma 5.2 and Corollary 5.3 of \cite{Dsystems} hold without any modifications for the system $\vec\eta$ considered here.

What does change in the present context is Lemma 5.4 of \cite{Dsystems}, which we reformulate as follows.

\begin{lemma}\label{S3:L8-alt1} 
Let $\Omega={\mathbb R}^n_+$ and let ${\mathcal L}$ be a block-from operator as above.
Suppose $\vec\eta$ is a weak solution of \eqref{system2} in $\Omega$.  For a fixed (sufficiently large) $a>0$, consider an arbitrary Lipschitz function $\hbar:{\mathbb R}^{n-1}\to \mathbb R$ such that
\begin{equation}
\|\nabla \hbar\|_{L^\infty}\le 1/a,\qquad \hbar(x)\ge 0\text{ for all }x\in{\mathbb R}^{n-1}.\label{hbarprop}
\end{equation}
Then for sufficiently large $b=b(a)>0$ we have the following. For an arbitrary surface ball $\Delta_r\subset{\mathbb R}^{n-1}$ of radius $r$ such that at least one point of $\Delta_r$
the inequality $\hbar(x)\le 2r$ holds we have the following estimate for all $m=1,2,\dots,n-1$ and an arbitrary $\vec{c}=(c_1,c_2,\dots,c_{n-1})\in\mathbb R$:
\begin{align}\label{TTBBMM}
&\sum_{m<n}\int_{1/6}^6\int_{\Delta_r}\big|\eta_m\big(x,\theta\hbar(x)\big)-c_m\big|^2\,dx\,d\theta
\leq C\|S_b(\vec\eta)\|_{L^2(\Delta_{2r})}
\|{N}_a(\vec\eta-\vec{c})\|_{L^2(\Delta_{2r})}
\nonumber\\+&C(k)(\|{N}_a(\vec\eta-\vec{c})\|_{L^2(\Delta_{2r})}^2+\|{N}_a(\vec\eta)\|_{L^2(\Delta_{2r})}^2+\|N_a(V)\|^2_{L^2(\Delta_{2r})})\nonumber\\
+&C\|S_b(\vec\eta)\|^2_{L^2(\Delta_{2r})}+\frac{C}{r}\iint_{\mathcal{K}}|\vec\eta-\vec{c}|^{2}\,dX,
\end{align}
for some $C\in(0,\infty)$ that only depends on $a,\Lambda,n$ but not on $\vec\eta$, $\vec{c}$ or $\Delta_r$
and $C(k)>0$ depends only on $k$ chosen to define $B_\parallel$, the Carleson norm $\|\mu\|_{Carl}$ and has the property that $C(k)\to 0$ as $k\to\infty$.

 Here $\mathcal{K}$ denotes a region inside $\Omega$ such that its diameter, 
distance to the graph $(\cdot,\hbar(\cdot))$, and distance to $\Delta_r$, are all comparable to $r$. 
Also, the cones used to define the square and nontangential 
maximal functions in this lemma have vertices on $\partial\Omega$.

Moreover, the term $\frac{C}{r}\displaystyle\iint_{\mathcal{K}}|\vec\eta-\vec{c}|^2\,dX$ appearing 
in \eqref{TTBBMM} may be replaced by the quantity
\begin{equation}\label{Eqqq-25}
Cr^{n-1}\left|(\vec\eta-\vec{c})(A_r)\right|^2,
\end{equation}
where $A_r$ is any point inside $\mathcal{K}$ (usually called a corkscrew point of $\Delta_r$).
\end{lemma}

\begin{proof} Let $\Delta_r$ be as in the statement of our Lemma. and assume that $(q,0)$ in the center of our ball. Let $\zeta$ be a smooth cutoff function of the form $\zeta(x,t)=\zeta_{0}(t)\zeta_{1}(x)$ where
\begin{equation}\label{Eqqq-27}
\zeta_{0}= 
\begin{cases}
1 & \text{ in } (-\infty, r_0+r], 
\\
0 & \text{ in } [r_0+2r, \infty),
\end{cases}
\qquad
\zeta_{1}= 
\begin{cases}
1 & \text{ in } \Delta_{r}(q), 
\\
0 & \text{ in } \mathbb{R}^{n}\setminus \Delta_{2r}(q)
\end{cases}
\end{equation}
and
\begin{equation}\label{Eqqq-28}
r|\partial_{t}\zeta_{0}|+r|\nabla_{x}\zeta_{1}|\leq c
\end{equation}
for some constant $c\in(0,\infty)$ independent of $r$. Here 
$r_0=6\sup_{x\in \Delta_r(q)}\hbar(x)$. Observe that our assumptions imply that
$$0\le r_0-\theta\hbar(x)\le  r_0 \lesssim r,\qquad \mbox{for all }x\in \Delta_{2r}(q),$$
for $\theta\in (1/6,6)$.

Our goal is to control the $L^2$ norm of $\eta_m\big(\cdot,\theta\hbar(\cdot)\big)-c_m$.  We fix $m\in\{1,\dots,n-1\}$ and proceed to estimate
\begin{align}
&\hskip -0.20in
\int_{\Delta_{r}(q)}(\eta_m(x,\theta\hbar(x))-c_m)^2\,dx \le \mathcal I:=\int_{\Delta_{2r}(q)}(\eta_m(x,\theta\hbar(x))-c_m)^2\zeta(x,\theta\hbar(x))\,dx
\nonumber\\[4pt]
&\hskip 0.70in
=-\iint_{\mathcal S(q,r,r_0,\theta\hbar)}\partial_{t}\left[(\eta_m(x,t)-c_m)^2\zeta(x,t)\right]\,dt\,dx,
\nonumber
\end{align}
where $\mathcal S(q,r,r_0,\theta\hbar)=\{(x,t):x\in \Delta_{2r}(q)\mbox{ and }\theta\hbar(x)<t<r_0+2r\}$. Hence:

\begin{align}\nonumber
&\hskip 0.10in
\mathcal I \le-2\iint_{\mathcal S(q,r,r_0,\theta\hbar)}(\eta_m-c_m)\partial_{t}(\eta_m-c_m)\zeta\,dt\,dx  
\\[4pt]
&\hskip 0.70in
\quad-\iint_{\mathcal S(q,r,r_0,\theta\hbar)}(\eta_m-c_m)^2(x,t)\partial_{t}\zeta\,dt\,dx
=:\mathcal{A}+IV.\label{u6tg}
\end{align}
We further expand the term $\mathcal A$ as a sum of three terms obtained 
via integration by parts with respect to $t$ as follows:
\begin{align}\label{utAA}
\mathcal A &=-2\iint_{\mathcal S(q,r,r_0,\theta\hbar)}(\eta_m-c_m)\partial_{t} 
(\eta_m-c_m)(\partial_{t}t)\zeta\,dt\,dx 
\nonumber\\[4pt]
&=2\iint_{\mathcal S(q,r,r_0,\theta\hbar)}\left|\partial_{t}\eta_m\right|^{2}t\zeta\,dt\,dx 
\nonumber\\[4pt]
&\quad +2\iint_{\mathcal S(q,r,r_0,\theta\hbar)}(\eta_m-c_m)\partial^2_{tt}(\eta_m-c_m)t\zeta\,dt\,dx 
\nonumber\\[4pt]
&\quad +2\iint_{\mathcal S(q,r,r_0,\theta\hbar)}(\eta_m-c_m)\partial_{t}\eta_m\,t\partial_{t}\zeta\,dt\,dx
\nonumber\\[4pt]
&=:I+II+III.
\end{align}

We start by analyzing the term $II$. As the $\eta_m$ solve the PDE \eqref{system2} then we have for $\eta_m-c_m$:
$${\mathcal L}(\eta_m-c_m)=\sum_{i,j=1}^{n-1}\partial_i((\partial_m b_{ij})(v_j+\eta_j))$$
and thus
\begin{equation}\label{S3:T8:E01-x}
\partial^2_{tt}(\eta_m-c_m)=\sum_{i,j=1}^{n-1}\partial_i((\partial_m b_{ij})(v_j+\eta_j))-\sum_{i,j=1}^{n-1}\partial_i(b_{ij}\partial_j(\eta_m-c_m)).
\end{equation}
In turn, this permits us to write the term $II$ as
\begin{align}
II &=-2\sum_{i,j<n}\iint_{\mathcal S(q,r,r_0,\theta\hbar)}(\eta_m-c_m)\partial_{i}\left({b}_{ij}\partial_{j}\eta_m\right)t\zeta\,dt\,dx
\nonumber\\[4pt]
&\quad+2\sum_{i,j<n}\iint_{\mathcal S(q,r,r_0,\theta\hbar)}(\eta_m-c_m)\partial_{i}\left((\partial_m b_{ij})(v_j+\eta_j)\right)t\zeta\,dt\,dx 
\nonumber\\[4pt]
\end{align}
Integrating both terms by parts w.r.t. $\partial_i$ then yields
\begin{align}
&=2\sum_{i,j<n}\iint_{\mathcal S(q,r,r_0,\theta\hbar)}b_{ij}\partial_i\eta_m\partial_{j}\eta_m\,t\zeta\,dt\,dx
\nonumber\\[4pt]
&+2\sum_{i,j<n}\iint_{\mathcal S(q,r,r_0,\theta\hbar)}b_{ij}(\eta_m-c_m)\partial_{j}(\eta_m)\,t(\partial_i\zeta)\,dt\,dx
\nonumber\\[4pt]
&-2\sum_{i,j<n}\iint_{\mathcal S(q,r,r_0,\theta\hbar)}(\partial_m b_{ij})(\partial_i\eta_m)(v_j+\eta_j)t\zeta\,dt\,dx 
\nonumber\\[4pt]
&-2\sum_{i,j<n}\iint_{\mathcal S(q,r,r_0,\theta\hbar)}(\partial_m b_{ij})(\eta_m-c_m)(v_j+\eta_j)t(\partial_i \zeta)\,dt\,dx 
\nonumber\\[4pt]
&\quad-2\sum_{i>0}\int_{\partial\mathcal S(q',r,r_0,\theta\hbar)}(\mbox{boundary terms})t\zeta\nu_i\,dS
\nonumber\\[4pt]
&=:II_{1}+II_{2}+II_{3}+II_{4}+II_{5}.\label{TFWW}
\end{align}
The boundary integral (term $II_5$) vanishes everywhere except on the graph of the function $\theta\hbar$ which implies that
\begin{align}
|II_5|&\le C\sum_{i,j<n}\int_{\Delta_{2r}(q)}|(\eta_m-c_m)(x,\theta\hbar(x)))\nabla_x(\eta_m)(x,\theta\hbar(x))\hbar(x)\zeta(x,\theta\hbar(x)) \nu_i|dS.\nonumber\\
&\,+ C\sum_{i,j<n}\int_{\Delta_{2r}(q)}|\partial_m b_{ij}||(\eta_m-c_m)(x,\theta\hbar(x)))(\eta_j+v_j)(x,\theta\hbar(x))\hbar(x)\zeta(x,\theta\hbar(x)) \nu_i|dS.\nonumber\\
&\le \frac12\int_{\Delta_{2r}(q)}(\eta_m(x,\theta\hbar(x))-c_m)^2\zeta(x,\theta\hbar(x))\,dx\nonumber\\&\quad+C'
\int_{\Delta_{2r}(q)}|\nabla_x \eta_m(x,\theta\hbar(x))|^2|\hbar(x)|^2\,dx
+C(k)\int_{\Delta_{2r}(q)}(|\vec\eta|^2+|V|^2)\,dx\\\nonumber
&=\frac12\mathcal I+II_6+II_7.
\end{align}
Here we have used the Cauchy-Schwarz for the first two terms and then the fact that $|\nabla_x B_\parallel |t\le \sqrt{C(k)}$ with $C(k)\to 0$ as $k\to\infty$ which is a consequence of \eqref{ss2} and how we have defined $B_\parallel$.

 We can hide the term $\frac12\mathcal I$ on the lefthand side of \eqref{u6tg}, while the second term after integrating $II_6$ in $\theta$ becomes:
\begin{align}
\int_{1/6}^6|II_6|\,d\theta&\le C\int_{1/6}^6\int_{\Delta_{2r}(q)}|\nabla \eta_m(x,\theta\hbar(x))|^2 |\hbar(x)|^2dxd\theta.\nonumber\\
&\lesssim\iint_{\Delta_{2r}(q)\times[0,r_0]}|\nabla \eta_m|^2t\,dt\,dx\lesssim\|S_b(\eta_m)\|^2_{L^2(\Delta_{2r})}.
\end{align}
The term $II_7$ can be estimated using the nontangential maximal function and is bounded by 
\begin{equation}
|II_7|\lesssim C(k)\left(\|N_a(\vec\eta)\|_{L^2(\Delta_{2r})}^2+\|N_a(V)\|_{L^2(\Delta_{2r})}^2\right)
\end{equation}
where in the last line we have used \eqref{sss1}.

Some of the remaining (solid integral) terms that are of the same type we estimate together. Firstly, we have 
\begin{equation}\label{Eqqq-29}
|I+II_1|\lesssim|S_b(\vec{\eta})\|^2_{L^2(\Delta_{2r})}.
\end{equation}
Here, the estimate holds even if the square function truncated at a hight $O(r)$.
Next, since $r|\nabla\zeta|\le c$, if the derivative falls on the cutoff function $\zeta$ we have
\begin{align}\label{TDWW}
|II_2+III| &\lesssim \iint_{[0,2r]\times \Delta_{2r}}\left|\nabla \vec\eta\right||\vec\eta-\vec{c}|\frac{t}{r}\,dt\,dx
\nonumber\\[4pt]
&\le \left(\iint_{[0,2r]\times \Delta_{2r}}|\vec\eta-\vec{c}|^{2}\frac{t}{r^{2}}\,dt\,dx\right)^{1/2} 
\|S^{2r}_b(\vec\eta)\|_{L^2(\Delta_{2r})} 
\nonumber\\[4pt]
&\lesssim\|S_b(\vec\eta)\|_{L^2(\Delta_{2r})}\|\tilde{N}_a(\vec\eta-\vec{c})\|_{L^2(\Delta_{2r})}.
\end{align}

The Carleson condition for $|\nabla B_\parallel|^2t$ and the Cauchy-Schwarz inequality imply
\begin{equation}\nonumber
|II_3| \lesssim C(k)
\|S_b(\vec\eta)\|_{L^2(\Delta_{2r})}(\|{N}_a(\vec\eta)\|_{L^2(\Delta_{2r})}^2+\|N_a(V)\|^2_{L^2(\Delta_{2r})})^{1/2}.
\end{equation}
For the term $II_4$ we use both that $r|\nabla\zeta|\le c$ and $|\nabla_x B_\parallel |t\le \sqrt{C(k)}$. It follows that  
\begin{equation}\nonumber
|II_4| \lesssim C(k) \iint_{[0,2r]\times \Delta_{2r}}|\vec\eta-\vec{c}||V+\vec\eta|\frac{t}{r^{2}}\,dt\,dx.
\end{equation}
An application of Cauchy-Schwarz inequality then implies that
\begin{equation}\nonumber
|II_4| \lesssim C(k) (\|{N}_a(\vec\eta-\vec{c})\|_{L^2(\Delta_{2r})}^2+\|{N}_a(\vec\eta)\|_{L^2(\Delta_{2r})}^2+\|N_a(V)\|^2_{L^2(\Delta_{2r})}).
\end{equation}

Finally, the interior term $IV$, which arises from the fact that $\partial_{0}\zeta$ vanishes on the set
$(-\infty,r_0+r)\cup(r_0+2r,\infty)$ may be estimated as follows:
\begin{equation}\label{Eqqq-31}
|IV|\lesssim\frac{1}{r}\iint_{\Delta_{2r}(q)\times [r_0+r,r_0+2r]}|\vec\eta-\vec{c}|^{2}\,dt\,dx.
\end{equation}
We put together all terms and integrate in $\theta$. The above analysis ultimately yields \eqref{TTBBMM}.
Finally, the last claim in the statement of the lemma that we can use \eqref{Eqqq-25} on the righthand side instead of the solid integral is a consequence of the Poincar\'e's inequality (see \cite{DHM} for detailed discussion).
\end{proof}

We now make use of Lemma~\ref{S3:L8-alt1}, involving the stopping time Lipschitz functions 
$\theta h_{\nu,a}(w)$, in order to obtain a localized good-$\lambda$ inequality. We omit the proof as it is identical to the one given in \cite{DHM}. Here
$$Mf(x):=\sup_{r>0}\fint_{|x-z|<r}|f(z)|\,dz\mbox{ for }x\in{\mathbb{R}}^{n-1},$$ 
denotes the standard Hardy-Littlewood maximal function on $\partial\mathbb R^n_+=\mathbb R^{n-1}$. 

\begin{lemma}\label{LGL-loc-alt1} Let $\mathcal L$ be an operator as in \eqref{S3:L8-alt1}.
Consider any boundary ball $\Delta_d=\Delta_d(q)\subset {\mathbb R}^{n-1}$, let $A_d=(q,d/2)$ be its corkscrew point and let
\begin{equation}
\nu_0=|\vec\eta(A_d)|.
\end{equation}
Then for each $\gamma\in(0,1)$ there exists a constant $C(\gamma)>0$ 
such that $C(\gamma)\to 0$ as $\gamma\to 0$ and with the property that for each $\nu>2\nu_0$ and 
any $\eta$ that satisfies \eqref{system2} there holds 
\begin{align}\label{eq:gl2}
&\hskip -0.20in 
\Big|\Big\{x\in {\mathbb R}^{n-1}:\,{N}_a(\eta\chi_{T(\Delta_d)})>\nu,\,(M(S^2_b(\eta)))^{1/2}\leq\gamma\nu,
\nonumber\\[4pt] 
&\hskip 0in
\big(C(k)[M({N}_a^2(\eta\chi_{T(\Delta_d)}))+ M({N}_a^2(V))]\big)^{1/2}\leq\gamma\nu,
\nonumber\\[4pt] 
&\hskip 0in
\big(M(S^2_b(\eta))M({N}_a^2(\eta\chi_{T(\Delta_d)}))\big)^{1/4}\leq\gamma\nu\Big\}\Big|
\nonumber\\[4pt] 
&\hskip 0.50in
\quad\le C(\gamma)\left|\big\{x\in{\mathbb R}^{n-1}:\,{N}_a(\eta\chi_{T(\Delta_d)})(x)>\nu/32\big\}\right|.
\end{align}
Here $\chi_{T(\Delta_d)}$ is the indicator function of the Carleson region $T(\Delta_d)$ and the square function
$S_b$ in \eqref{eq:gl2} is truncated at the height $2d$. Similarly, the Hardy-Littlewood maximal operator $M$
is only considered over all balls $\Delta'\subset\Delta_{md}$ for some enlargement constant $m=m(a)\ge 2$.
\end{lemma}

Finally we have the following proposition, again by the same argument as in \cite{DHM}.

\begin{proposition}\label{S3:C7-alt1} Under the assumptions of Lemma \ref{LGL-loc-alt1}, for sufficiently large $k$ we have that for any $p>0$ and $a>0$ there exists an integer $m=m(a)\ge 2$ and a finite 
constant $C=C(n,p,a,\|\mu\|_{Carl})>0$ such that for all balls $\Delta_d\subset{\mathbb R}^{n-1}$ we have
\begin{equation}\label{S3:C7:E00ooloc}
\|\tilde{N}^r_a(\vec\eta)\|_{L^{p}(\Delta_d)}\le C\|S^{2r}_a(\vec\eta)\|_{L^{p}(\Delta_{md})}
+C\|\tilde{N}_a(V)\|_{L^{p}(\Delta_{md})}
+Cd^{(n-1)/p}|\vec\eta(A_d)|,
\end{equation}
where $A_d$ denotes the corkscrew point of the ball $\Delta_d$.

We also have a global estimate for any $p>1$ and $a>0$. There exists a 
constant $C>0$ such that 
\begin{equation}\label{S3:C7:E00oo}
\|{N}_a(\vec\eta)\|_{L^{p}({\mathbb R}^{n-1})}\le C\|S_a(\vec\eta)\|_{L^{p}({\mathbb R}^{n-1})}+C\|\nabla_x f\|_{L^{p}({\mathbb R}^{n-1})}.
\end{equation}
Here we have used the estimate \eqref{sss1}.

\end{proposition}
We can now combine Proposition \eqref {S3:C7:E00oo} with estimate \eqref{system6}. It follows that
\begin{eqnarray}\label{sss33}
\|{N}(\vec\eta)\|_{L^{2}({\mathbb R}^{n-1})}&\le& C\|S(\vec\eta)\|_{L^{2}({\mathbb R}^{n-1})}+C\|\nabla_x f\|_{L^{2}({\mathbb R}^{n-1})}\\\nonumber
&\le& C\|\nabla_x f\|_{L^{2}} + C(k) \|{N}(\vec\eta)\|_{L^{2}({\mathbb R}^{n-1})}.
\end{eqnarray}
For $k$ chosen so large that the constant $C(k)<1/2$ we then obtain
\begin{equation}\label{S3:C7:E00ooss}
\|{N}(\vec\eta)\|_{L^{2}({\mathbb R}^{n-1})}\le 2C\|\nabla_x f\|_{L^{2}({\mathbb R}^{n-1})}.
\end{equation}
Then by \eqref{sss1} we obtain a similar estimate for $W=(w_1,w_2\dots,w_{n-1})$:
\begin{equation}\label{S3:C7:E00ooss2}
\|{N}(W)\|_{L^{2}({\mathbb R}^{n-1})}\lesssim\|\nabla_x f\|_{L^{2}({\mathbb R}^{n-1})},
\end{equation}
which would imply that the $L^2$ Regularity problem for $\mathcal L$ is solvable if we can also establish nontangential estimates for $w_n=\partial_tu$. 

We establish an analogue of  Lemma \ref{S3:L8-alt1} for the function $w_n-c$ for an arbitrary $c\in\mathbb R$. Clearly, we have that
\begin{equation}\label{system33}
{\mathcal L}(w_n-c)=\sum_{i,j=1}^{n-1}\partial_i((\partial_t b_{ij})w_j).
\end{equation}
It follows that the calculation of Lemma \ref{S3:L8-alt1} can be followed step by step. The only difference is that we cannot claim any smallness of the Carleson measure as the measure we obtain is that of $|\partial_t B_\parallel|^2t$ which is not small for any $k$. This will not be
a problem since the terms that come with it contain $N(W)$ for which we already have required estimates. Hence we can prove that  

\begin{align}\label{TTBBMMx}
&\hskip-10mm\int_{1/6}^6\int_{\Delta_r}\big|w_n\big(x,\theta\hbar(x)\big)-c\big|^2\,dx\,d\theta
\leq C\|S_b(w_n)\|_{L^2(\Delta_{2r})}
\|{N}_a(w_n-c)\|_{L^2(\Delta_{2r})}
\nonumber\\+&C((\|{N}_a(w_n-c)\|_{L^2(\Delta_{2r})}+\|S_b(w_n)\|_{L^2(\Delta_{2r})})\|{N}_a(W)\|_{L^2(\Delta_{2r})}+\|N_a(W)\|^2_{L^2(\Delta_{2r})})\nonumber\\
+&C\|S_b(w_n)\|^2_{L^2(\Delta_{2r})}+\frac{C}{r}\iint_{\mathcal{K}}|w_n-c|^{2}\,dX,
\end{align}
From this we can obtain a good-lambda inequality and eventually a global estimate as before in the form of
\begin{equation}\label{S3:C7:E00oozz}
\|N(w_n)\|_{L^p(\mathbb R^{n-1})}\le C\|S(w_n)\|_{L^p(\mathbb R^{n-1})}+C\|N(W)\|_{L^p(\mathbb R^{n-1})}.
\end{equation}
It remains to prove square functions estimates for $S(w_n)$. But since we have estimates for $S(W)$ the only remaining term that needs an estimate is
$\int_{\mathbb R^n_{+}}|\partial_{tt}u|^2tdt\,dx$. Since $\mathcal Lu=0$ this PDE implies that
\begin{equation}
\int_{\mathbb R^n_{+}}|\partial_{tt}u|^2tdt\,dx=\sum_{i,j,s,r<n}\partial_i(b_{ij}\partial_ju)\partial_s(b_{sr}\partial_ru)t\,dt\,dx
\end{equation}
$$\le C\|S(W)\|^2_{L^2(\mathbb R^{n-1})}+C\int_{\mathbb R^n_{+}}|\nabla_x B_\parallel|^2|W|^2t\,dt\,dx$$
$$\le C\|S(W)\|^2_{L^2(\mathbb R^{n-1})}+C(k)\|N(W)\|^2_{L^2(\mathbb R^{n-1})}.$$
It follows that $\|N(w_n)\|_{L^2(\mathbb R^{n-1})}\le C\|N(W)\|_{L^2(\mathbb R^{n-1})}$ and hence by \eqref{S3:C7:E00ooss2} the Regularity problem in $L^2$ for $\mathcal L$ is solvable on $\mathbb R^{n}_+$.
As this also implies solvability for $\mathcal L_0$, this completes the argument.

\vglue5mm


\begin{thebibliography}{30}

\bibitem{Dahl} B. E. J. Dahlberg, {\it Poisson semigroups and singular integrals},
Proc. Amer. Math. Soc. 97(1) (1986), 41–48.

\bibitem{Dsystems} M. Dindo\v{s}, {\it The $L^p$ Dirichlet and regularity problems for second order elliptic systems with application to the Lamé system}, Comm. Partial Differential Equations 46 (2021), no. 9, 1628--1673. 

\bibitem{DHM} M. Dindoš, S. Hwang, M. Mitrea, {\it The $L^p$ Dirichlet boundary problem for second order elliptic systems with rough coefficients} Trans. Amer. Math. Soc. 374 (2021), no. 5, 3659–3701.

\bibitem{DK} M. Dindo\v{s}, J. Kirsch, {\it The regularity problem for
elliptic operators with boundary data in Hardy-Sobolev space $H^1$},  Math. Res. Lett., Vol. 19, (2012), no.3, 699-717.

\bibitem{DLP} M. Dindo\v{s}, J. Li, J. Pipher, {\it The p-ellipticity condition for second order elliptic systems and applications to the Lamé and homogenization problems}, J. Differential Equations 302 (2021), 367--405.

\bibitem{DPP} M. Dindo\v{s}, S. Petermichl, J. Pipher, {\it The $L^p$
Dirichlet problem for second order elliptic operators and a
$p$-adapted square function},  J. Funct. Anal. 249, (2007) 372--392.

\bibitem{DPcplx}
M. Dindo\v{s}, J. Pipher,
{\it Boundary value problems for second-order elliptic operators with complex coefficients}, Anal. PDE 13 no. 6, 1897--1938.
       
\bibitem{DP}
M. Dindo\v{s}, J. Pipher,
{\it Regularity theory for solutions to second order elliptic operators with complex coefficients and the $L^p$ Dirichlet problem},  
Adv. Math. 341 (2019), 255--298.


\bibitem{DPR} M. Dindo\v{s}, J. Pipher, D. Rule, {\it Boundary value problems for second-order elliptic operators satisfying a Carleson condition}, Comm. Pure Appl. Math. 70 (2017), no. 7, 1316--1365.
    
\bibitem{DR} M. Dindo\v{s}, D. Rule, {\it Elliptic equations in a plane
satisfying the Carleson measure condition}, Rev. M. Iber., 26 (2010), no. 3, 1013--1034.
    
\bibitem{F} J. Feneuil, A change of variables of Dahlberg-Kenig-Pipher operators, arXiv:2106.13152.
    
\bibitem{HKMP1}  S. Hofmann, C. Kenig, S. Mayboroda, J. Pipher, {\it Square function/non-tangential maximal function estimates and the Dirichlet problem for non-symmetric elliptic operators},  J. Amer. Math. Soc. 28 (2015), no. 2, 483--529.

\bibitem{HKMP2}  S. Hofmann, C. Kenig, S. Mayboroda, J. Pipher, {\it The Regularity problem for second order elliptic operators with complex-valued bounded measurable coefficients},  Math. Ann. 361 (2015), no. 3-4, 863--907.

\bibitem{HLM} S. Hofmann, P. Le, A. Morris, {\it Carleson measure estimates and the Dirichlet problem for degenerate elliptic equations},  Anal. PDE 12 (2019), no. 8, 2095–2146.

\bibitem{HMM} S. Hofmann; S. Mayboroda, M. Mourgoglou, {\it Layer potentials and boundary value problems for elliptic equations with complex
$L^\infty$ coefficients satisfying the small Carleson measure norm condition}, Adv. Math. 270 (2015), 480–564.

\bibitem{HMMTZ} S. Hofmann, J-M. Martell, S. Mayboroda, T. Toro, Z. Zhao, {\it Uniform rectifiability and elliptic operators satisfying a Carleson measure condition}, Geom. Funct. Anal. 31 (2021), no. 2, 325–401.
        
\bibitem{KKPT}  C. Kenig, B. Kirchheim, J. Pipher, T. Toro, {\it  Square functions and the $A_\infty$ property of elliptic measures}, J. Geom. Anal. 26 (2016), no. 3, 2383--2410.
    
\bibitem{KKPT2}  C. Kenig, H. Koch, J. Pipher, T. Toro, {\it A new approach to absolute continuity of elliptic measure, with applications to non-symmetric equations}, Adv. Math. 153 (2000), no. 2, 231--298. 

\bibitem{KP} C. Kenig, J. Pipher, {\it The Dirichlet problem for elliptic equations with drift terms}, Publ. Mat.  45  (2001),  no. 1, 199--217.

\bibitem{KP1} C. Kenig, J. Pipher,  {\it The Neumann problem for elliptic equations with nonsmooth coefficients} Invent. Math. 113 (1993), no. 3, 447–509.

\bibitem{KP2} C. Kenig, J. Pipher, {\it The Neumann problem for elliptic equations with nonsmooth coefficients. II. A celebration of John F. Nash, Jr.}, Duke Math. J. 81 (1995), no. 1, 227--250.

\bibitem{KR} C. Kenig, D. Rule, {\it The Regularity and Neumann problem for non-symmetric elliptic operators},
Trans. Amer. Math. Soc. 361 (2009), no. 1, 125--160.

\bibitem{MPT} M. Mourgoglou, B. Poggi, X. Tolsa, {\it $L^p$ solvability of the Poisson-Dirichlet problem and its applications to the Regularity problem}, arXiv:2207.10554.

\bibitem{MT} M. Mourgoglou, X. Tolsa, {\it The Regularity problem for the Laplace equation in rough domains}, arXiv:2110.02205. 

\bibitem{N} J. Ne\v{c}as, {\it Les méthodes directes en théorie des équations elliptiques}, (French) Masson et Cie, Éditeurs, Paris; Academia, Éditeurs, Prague 1967 351 pp.

\bibitem{Sa} D. Sarason, {\it Functions of vanishing mean oscillation} Trans. Amer. Math. Soc. 207 (1975), 391–405.

\bibitem{S} Z. Shen, {\it Extrapolation for the $L^p$ Dirichlet problem in Lipschitz domains}, Acta Math. Sin. (Engl. Ser.) 35 (2019), no. 6, 1074--1084.



\end{thebibliography}
\end{document}